\documentclass[10pt,a4paper]{amsart}
\usepackage{verbatim}

\usepackage[T1]{fontenc}
\usepackage{graphicx}
\usepackage{enumerate}
\usepackage{amsmath,amsfonts,amssymb}
\usepackage{color}
\usepackage{cite}
\definecolor{citation}{rgb}{0.2,0.58,0.2} 
\definecolor{formula}{rgb}{0.1,0.2,0.6}
\definecolor{url}{rgb}{0.3,0,0.5}
\usepackage{pgf,tikz}
\usepackage{mathrsfs}
\usepackage{fancyhdr}

\usepackage{dsfont}

\newcommand{\reqnomode}{\tagsleft@false}

\usepackage[colorlinks,pdfpagelabels,pdfstartview = FitH,bookmarksopen = true,bookmarksnumbered = true,urlcolor=url,linkcolor = formula,plainpages = false,hypertexnames = false,citecolor = citation] {hyperref}

\def\dx{\,{\rm d}x}

\def \d{\,{\rm d}}
\def \diver{\,{\rm div}}
\def\dist{\,{\rm dist}}
\def\eup{\,{\rm ess\, sup}}
\def\eif{\,{\rm ess\, inf}}
\def\supp{\,{\rm supp}}
\def\diam{\,{\rm diam}}

\allowdisplaybreaks
\makeatletter
\DeclareRobustCommand*{\bfseries}{%
  \not@math@alphabet\bfseries\mathbf
  \fontseries\bfdefault\selectfont
  \boldmath
}

\DeclareMathOperator*{\osc}{osc}

\makeatother

\newlength{\defbaselineskip}
\newcommand{\setlinespacing}[1]
           {\setlength{\baselineskip}{#1 \defbaselineskip}}
\newcommand{\mint}{\mathop{\int\hskip -1,05em -\, \!\!\!}\nolimits}

\newtheorem{theorem}{Theorem}
\newtheorem{corollary}{Corollary}[section]
\newtheorem{definition}{Definition}
\newtheorem{remark}{Remark}[section]
\newtheorem{lemma}{Lemma}[section]
\newtheorem{proposition}{Proposition}[section]
\newcommand{\N}{\mathbb{N}}
\numberwithin{equation}{section}
\allowdisplaybreaks[4]

\newcommand{\ic}[1]{\textcolor{teal}{#1}}
\newcommand{\rr}{\varrho}
\newcommand{\snr}[1]{\lvert #1\rvert}
\newcommand{\nr}[1]{\lVert #1 \rVert}

\newcommand{\data}{\textit{\texttt{data}}}

\title[
Removable sets for solutions of non-uniformly elliptic equations]{Removable sets in non-uniformly elliptic problems
}
\author{Iwona Chlebicka} \address{Iwona Chlebicka\\Faculty of Mathematics, Informatics and Mechanics, University of Warsaw\\ul. Banacha 2, 02-097 Warsaw, Poland} \email{\texttt{iskrzypczak@mimuw.edu.pl}}
\author{Cristiana De Filippis}  \address{Cristiana De Filippis\\Mathematical Institute, University of Oxford\\ Andrew Wiles Building, Radcliffe Observatory Quarter, Woodstock Road, Oxford, OX26GG, Oxford, United Kingdom} \email{\texttt{Cristiana.DeFilippis@maths.ox.ac.uk}}

\begin{document}

\subjclass[2010]{35J60, 35J70\vspace{1mm}} 

\keywords{Measure data problems, Obstacle problem, Potential estimates, Removable sets\vspace{1mm}}

\thanks{{\it Acknowledgements.}\ I. Chlebicka is supported by NCN grant no. 2016/23/D/ST1/01072. C. De Filippis is supported by the Engineering and Physical Sciences Research Council (EPSRC): CDT Grant Ref. EP/L015811/1. 
\vspace{1mm}}

\maketitle

\begin{abstract}
 We analyze fine properties of solutions to quasilinear elliptic equations with double phase structure and characterize, in the terms of intrinsic Hausdorff measures, the removable sets for H\"older continuous solutions. 
\end{abstract}
\vspace{3mm}

\setlinespacing{1.08}

\newcommand{\R}{\mathds{R}^n}
\section{Introduction}
Nonlinear version of the classical potential theory has been developed by several authors \cite{hekima,kima,kumi,kumi1,lima,ma,mik,tu}. Classical investigations on harmonic mappings has become naturally extended to the case of so-called $\mathcal{A}$-harmonic maps, i.e., continuous weak solutions to the equation
\begin{flalign}\label{i0}
-\diver A(x,Du)=0 \ \ \mbox{in} \ \ \Omega,
\end{flalign}
where $A(x,z)\cdot z\approx \snr{z}^{p}$ and $p\in (1,\infty)$. The prototypical example is given by the classical $p$-Laplacean equation
\begin{flalign*}
-\diver(\snr{Du}^{p-2}Du)=0 \ \ \mbox{in} \ \ \Omega.
\end{flalign*}
An interesting and delicate issue in this framework is quantifying the size of removable sets for solutions to \eqref{i0}, in terms of the $s$-dimensional Hausdorff measure for a suitable $s>n-p$. Precisely, when we consider a relatively closed subset $E\subset \Omega$ with $\mathcal{H}^{s}(E)=0$ and a map $u$, being a~solution to \eqref{i0} in $\Omega\setminus E$, the problem is to find a continuous extension $\tilde{u}$ which is $\mathcal{A}$-harmonic in the whole $\Omega$.  This issue has been studied in various types of nonstandard growth settings including the variable exponent and Orlicz spaces, which we summarize in the further parts of the introduction. We shall investigate this phenomenon for problems posed in the so-called double phase spaces. Throughout the paper, $\Omega$ is an open bounded subset of $\mathbb{R}^n$, $n\ge 2$.  Given the double phase energy density \begin{flalign}\label{H}
H(x,z):=\snr{z}^{p}+a(x)\snr{z}^{q} 
\end{flalign}
and the operator $\mathcal{A}_{H(\cdot)}$ defined on $W^{1,H(\cdot)}(\Omega)$ acting as
\begin{flalign}\label{calAH}
\langle\mathcal{A}_{H(\cdot)}v,w\rangle:=\int_{\Omega}A(x,Dv)\cdot Dw \ \dx\quad \text{for }\ w\in C^{\infty}_{c}(\Omega)
\end{flalign}
we investigate $\mathcal{A}_{H(\cdot)}$-harmonic maps, by which we mean energy solutions to the equation
\begin{flalign}\label{A0}
-\diver A(x,Du)=0 \ \ \mbox{in} \ \ \Omega.
\end{flalign}
The growth and coercivity conditions of the involved vector field $A$ are expressed in~\eqref{A} in terms of $H(\cdot)$ with non-negative, $\alpha$-H\"older continuous weight $a$ and $p,q$ satisfying the balance condition~\eqref{pq}. Note that the operator is non-uniformly elliptic since the weight $a$ is allowed to vanish. The space where the solutions are considered is the Musielak-Orlicz space $W^{1,H(\cdot)}(\Omega)$ defined in Section~\ref{ssec:setting}. Our main result describes the volume of removable sets for H\"older continuous $\mathcal{A}_{H(\cdot)}$-harmonic map in the terms of the intrinsic Hausdorff measures defined in Section~\ref{haus} with the use of $H_{\sigma}(\cdot)$ given by~\eqref{Hgamma}.
\begin{theorem}\label{T6}
Let assumptions \eqref{pq} and \eqref{A} be in force, $E\subset \Omega$ be a closed subset and $u\in C(\Omega)$ be a continuous solution to \eqref{A0} in $\Omega\setminus E$ such that, for all $x_{1}\in E$, $x_{2}\in \Omega$,
\begin{flalign*}
\snr{u(x_{1})-u(x_{2})}\le C_u\snr{x_{1}-x_{2}}^{\beta_{0}},
\end{flalign*}
for a positive, absolute constant $C_u$ and some $\beta_{0}\in (0,1]$.\\ If $\mathcal{H}_{H_{\sigma}(\cdot)}(E)=0$, for $\sigma:=1-\tfrac{\beta_{0}}{q}(p-1)$, then $u$ is a solution to \eqref{A0} in $\Omega$.
\end{theorem}
This matter received lots of attention in the past decades: see \cite{ca,koma,kizo} for the first works on the standard or power growth case and \cite{ono,hi} for problems involving also lower-order terms. More recently the $p$-Laplace structure has been relaxed to more flexible ones involving nonstandard growth conditions of various types. The related variational functionals include the variable-exponent integrand fundamental in modelling electrorheological fluids
\begin{flalign}\label{px}
w\mapsto \mathcal{F}_{p(\cdot)}(w,\Omega):=\int_{\Omega}\snr{Dw}^{p(x)} \ \dx,
\end{flalign}
the double phase energy describing strongly inhomogeneous materials
\begin{flalign}\label{dp-f}
w\mapsto \mathcal{F}_{H(\cdot)}(w,\Omega):=\int_{\Omega}\snr{Dw}^{p}+a(x)\snr{Dw}^{q} \ \dx,
\end{flalign}
as well as the so-called $\varphi$-functional, defined by means of an $N$-functions $\varphi$ cf.~\cite{cima}, and involved in the modelling of non-Newtonian fluids
\begin{flalign}\label{phi}
w\mapsto \mathcal{F}_{\varphi}(w,\Omega):=\int_{\Omega}\varphi(\snr{Dw}) \ \dx,
\end{flalign}or more generally
\begin{flalign}\label{phix}
w\mapsto \mathcal{F}_{\varphi(\cdot)}(w,\Omega):=\int_{\Omega}\varphi(x,\snr{Dw}) \ \dx,
\end{flalign}
see \cite{IC-pocket, IC-b, hahab} for more details. In the nonstandard growth framework, the problem of removability of sets have been studied in the case of \eqref{px} in \cite{fush,lat} and of~\eqref{phi} in \cite{chaly}. We shall provide the related result in the inhomogeneous setting \eqref{dp-f}.  A remarkable difference between our work and the aforementioned papers is that the main estimates involve the  intrinsic capacities and the intrinsic Hausdorff measures introduced recently in \cite{bahaha,demi} respectively. These novel concepts seem to be very natural in the general Musielak-Orlicz setting. The setting of the double phase spaces, which we employ, is of high interest recently. They appeared originally in the context of homogenization and the Lavrentiev phenomenon~\cite{zh}.  Recently, regularity theory in this setting is getting increasing attention \cite{bacomi-st,bacomi,comi,comib,deoh,depa,Ok}, see also \cite{depq,de,demi} for the manifold-constrained case and \cite{dark} for a reasonable survey on older results. The growth of the operator we investigate is trapped between two power-type functions following the ideas of~\cite{m1,m2}. The inhomogeneity of our setting results from the fact that the modulating coefficient $a$ can vanish: on $\{x\in \Omega\colon a(x)=0\}$, \eqref{H} shows $p$-growth, whereas on $\{x\in\Omega: a(x) > 0\}$, it is behaves like an $N$-function $\varphi(x,\cdot)$.  See comments around \eqref{adeg}-\eqref{andeg} below for further clarifications on this matter. It is the regularity of the weight function $a$ that dictates the ellipticity rate of the energy density indicating the range of parameters necessary to ensure good properties of the space such as density of smooth functions.  In fact, as far as a variable exponent is expected to be log-H\"older continuous, here the exponents should be close to each other. The optimal closeness condition in the case of $a\in C^{0,\alpha}(\Omega)$ is~\eqref{pq}, see~\cite{eslemi,comi} and Remark~\ref{rem:density}. The relation between the double phase space and the variable exponent one is exposed in~\cite{bacomi-st}.  More information on functional analysis of our setting can be found in Section~\ref{sec:prelim}. The non-uniform ellipticity of the operator entails essential structural difficulties in the proof of Theorem~\ref{T6} in comparison to the standard growth case. Nonetheless, the main idea is the same and it involves certain regularity properties of~solutions to the obstacle problem associated to~\eqref{A0}. We stress that the tools applied in this paper have been defined for more general structures than the one described in Section~\ref{dp}, i.e. for $H(\cdot)$ substituted by the so-called generalized Young functions $\varphi(\cdot)$. Hence, we expect that 
analogous results to those reported in Theorems~\ref{T6}-\ref{T4} hold for quasilinear equations modelled upon~\eqref{phix}. For more details and further extensions we refer the reader to \cite{bahaha,demi,hahale}. Since it can be of independent interest, we do not restrict ourselves to the regularity theory for solutions to the obstacle problem related to~\eqref{A0} necessary for the proof of  Theorem~\ref{T6}. In this place, besides~\cite{choe}, we shall mention that there are studies carried in various types of nonstandard growth settings starting from the variable exponent setting~\cite{ha-o,ele,ele-1,ele-2} and Orlicz~\cite{lib} as well as other~\cite{Ok-o1,Ok-o2}. We prove that solutions share the same features of the obstacle, that is higher integrability, boundedness, continuity, and H\"older continuity, respectively.  Let us present the obstacle problem we study. We consider the set
\begin{flalign}\label{con}
\mathcal{K}_{\psi,g}(\Omega):=\left\{ v\in W^{1,H(\cdot)}(\Omega)\colon v\ge \psi \ \  \mbox{a.e. in} \ \Omega \ \ \mbox{and} \ \ v-{g}\in W^{1,H(\cdot)}_{0}(\Omega)  \right\},
\end{flalign}
where $\psi\in W^{1,H(\cdot)}(\Omega)$ is the obstacle and $g \in W^{1,H(\cdot)}(\Omega)$ is the boundary datum. By a~solution to the obstacle problem we mean a function $v\in \mathcal{K}_{\psi,g}(\Omega)$ satisfying
\begin{flalign}\label{obs}
\int_{\Omega}A(x,Dv)\cdot D(w-v) \ \dx \ge 0 \ \ \mbox{for all } \ w\in \mathcal{K}_{\psi,g}(\Omega).
\end{flalign} 
By a supersolutions to \eqref{A0} we mean $\tilde{v}\in W^{1,H(\cdot)}(\Omega)$ satisfying
\begin{flalign}\label{sux}
\int_{\Omega}A(x,D\tilde{v})\cdot Dw \ \dx \ge0 \quad \mbox{for all non-negative} \ w\in W^{1,H(\cdot)}_{0}(\Omega).
\end{flalign}
Notice that a solution to problem \eqref{obs} is a supersolution to \eqref{A0}, we just need to test \eqref{obs} against $w:=v+\tilde{w}$, where $\tilde{w} \in W^{1,H(\cdot)}_{0}(\Omega)$ is any non-negative function. Since $\tilde{w}\in \mathcal{K}_{\psi,g}(\Omega)$,  the outcome is precisely the variational inequality \eqref{sux}. We provide the following result on existence and basic regularity for the obstacle problem. 
\begin{theorem}\label{T4}
Under assumptions \eqref{pq} and \eqref{A}{, let $\psi,g\in W^{1,H(\cdot)}(\Omega)$ be such that $\mathcal{K}_{\psi,g}(\Omega)\not =\emptyset$. Then, there exists a unique $v\in \mathcal{K}_{\psi,g}(\Omega)$, solution to the obstacle problem \eqref{obs}. Moreover, the following holds true.}
\begin{itemize}
    \item[-](Gehring's Theory). If $H(\cdot,D\psi)\in L^{1+\delta_{1}}_{\mathrm{loc}}(\Omega)$ for some $\delta_{1}>0$, then there exists an integrability threshold $ {\delta}_{0}\in (0,\delta_{1})$ {such that for any $\delta \in (0, {\delta}_{0})$, $H(\cdot,Dv)\in L^{{1+\delta}}_{\mathrm{loc}}(\Omega)$} and the following {reverse--H\"older--type}  inequality holds for all $B_{\rr}\Subset \Omega$:
    \begin{flalign}\label{in:ge}
\left(\mint_{B_{\rr/2}}H(x,Dv)^{1+\delta} \ \dx\right)^{\frac{1}{1+\delta}}\le c\left(\int_{B_{\rr}}H(x,D\psi)^{1+\delta} \ \dx\right)^{\frac{1}{1+\delta}}+c\mint_{B_{\rr}}H(x,Dv) \ \dx.
\end{flalign}
Here $c=c(\data, \nr{Dv}_{L^{p}(\Omega)},\nr{D\psi}_{L^{p}(\Omega)})$ and $ {\delta}_{0}= {\delta}_{0}(\data, \nr{Dv}_{L^{p}(\Omega)},\nr{D\psi}_{L^{p}(\Omega)})$.
\vspace{1mm}
\item[-](De Giorgi's Theory). If $\psi \in W^{1,H(\cdot)}(\Omega)\cap L^{\infty}(\Omega)$, then $v\in L^{\infty}_{\mathrm{loc}}(\Omega)$ and, for any open set $\tilde{\Omega}\Subset \Omega$,
\begin{flalign}\label{v-bounded}
\nr{v}_{L^{\infty}(\tilde{\Omega})}\le c(\data,\nr{Dv}_{L^{p}(\Omega)},\nr{H(\cdot,v)}_{L^{1}(\Omega)},\nr{\psi}_{L^{\infty}(\Omega)}).
\end{flalign}
\item[-]{(Continuity and $\mathcal{A}_{H(\cdot)}$-harmonicity). If $\psi\in W^{1,H(\cdot)}(\Omega)\cap C(\Omega)$, then $v$ is continuous and solves \eqref{A0} in the open set $\{x\in \Omega\colon v(x)>\psi(x)\}$.}
\vspace{1mm}
\item[-]{(H\"older regularity). If $\psi \in W^{1,H(\cdot)}(\Omega)\cap C^{0,\beta_{0}}(\Omega)$ for some $\beta_{0}\in (0,1]$, then $v\in C^{0,\beta_{0}}_{\mathrm{loc}}(\Omega)$ and, for all open sets $\tilde{\Omega}\Subset \Omega$, there holds
\begin{flalign}\label{v-Holder}
[v]_{0,\beta_{0};\tilde{\Omega}}\le c(\data,\nr{H(\cdot,Dv)}_{L^{1}(\Omega)},\nr{\psi}_{L^{\infty}(\Omega)},[\psi]_{0,\beta_{0}}).
\end{flalign}}
\end{itemize}
\end{theorem} 
Let us sum up the organization of the paper. Section~\ref{sec:prelim} introduces main assumptions and the functional setting. In Section~\ref{haus} we present the concept of intrinsic capacities and intristic Hausdorff--type measures both related to the energy density $H(\cdot)$. Section~\ref{o} is devoted to the study on the obstacle problem, while Section~\ref{sec:rem} to the proof of the main result on the removability.

\vspace{2mm}
\section{Preliminaries}\label{sec:prelim}
\subsection{Assumptions} Throughout the paper $\Omega \subset \mathbb{R}^{n}$, $n\ge 2$, is an open, bounded set. Let $A\in C(\Omega\times \mathbb{R}^{n},\mathbb{R}^{n})$ be a monotone vector field such that $z\mapsto A(\cdot,z)\in C^{1}(\mathbb{R}^{n}\setminus\{0\},\mathbb{R}^{n})$ and the following growth and coercivity assumptions hold true:
\begin{flalign}\label{A}
\begin{cases}
\ \snr{A(x,z)}+\snr{\partial A(x,z)}\snr{z}\le L\left(\snr{z}^{p-1}+a(x)\snr{z}^{q-1}\right),\\
\ \nu\left(\snr{z}^{p-2}+a(x)\snr{z}^{q-2}\right)\snr{\xi}^{2}\le \partial A(x,z)\xi\cdot \xi,\\
\ \snr{A(x_{1},z)-A(x_{2},z)}\le L\snr{a(x_{1})-a(x_{2})}\snr{z}^{q-1},
\end{cases}
\end{flalign}
whenever $x_{1}$, $x_{2}\in \Omega$, $z\in \mathbb{R}^{n}\setminus \{0\}$, $x\in \Omega$ and $\xi \in \mathbb{R}^{n}$. Here, $\partial$ denotes the partial derivative with respect to the gradient variable $z$, while the modulating coefficient $a\colon \Omega\to [0,\infty)$ is a~non-negative and $\alpha$-H\"older continuous function for some $\alpha\in (0,1]$. We assume that the exponents $p$, $q$ appearing in \eqref{H} and the H\"older continuity exponent $\alpha$ mentioned above satisfy the relations
\begin{flalign}\label{pq}
\frac{q}{p} \le 1+\frac{\alpha}{n} \quad \mbox{and} \quad 1<p < q< n.
\end{flalign}
For brevity we  collect the main parameters of the problem in the quantities
{\it \begin{flalign*}
\begin{cases}\ \data:=(n,\nu,L,p,q,[a]_{0,\alpha}),\\
\ \data_{\psi}:=(n,\nu,L,p,q,[a]_{0,\alpha}, \nr{H(\cdot,D\psi)}_{L^{1}(\Omega)}, \nr{\psi}_{L^{\infty}(\Omega)}),\\
\ \data_{u}:=(n,\nu,L,p,q,[a]_{0,\alpha}, \nr{H(\cdot,Du)}_{L^{1}(\Omega)},\nr{u}_{L^{\infty}(\Omega)}).
\end{cases}
\end{flalign*}}
\subsection{Notation}\label{sec:not}  We collect here basic remarks on the notation we use throughout the paper. Following a usual custom, we denote by $c$ a general constant larger than one. Different occurrences from line to line will be still denoted by $c$, while special occurrences will be denoted by $\tilde{c},\bar{c}$ or similarly. Relevant dependencies on parameters will be emphasized using parentheses, e.g. $c=c(n,p,q)$ means that $c$ depends on $n,p,q$. If $t \in \{p,q\}$, by $t'$ we mean the H\"older conjugate of $t$, i.e. $t'=t/(t-1),$ whereas by $t^{*}$ its Sobolev conjugate, i.e. $t^{*}=tn/(n-t)$ (recall $\eqref{pq}_{2}$). We denote by $B_{\rr}(x_{0}):=\left\{x\in \mathbb{R}^{n}\colon \snr{x-x_{0}}<\rr\right\}$ the open ball with center $x_{0}$ and radius $\rr>0$. When it is not important, or clear from the context, we shall omit denoting the center as follows: $B_{\rr}(x_{0})\equiv B_{\rr}$. Very often, when it is not otherwise stated, different balls will share the same center. When the ball $B$ is given we occasionally denote its radius as $\rr(B)$. With $U\subset \mathbb{R}^{n}$ being a~measurable set with finite and positive $n$-dimensional Lebesgue measure $\snr{U}>0$, and with $f\colon U\to \mathbb{R}^{k}$, $k\ge 1$ being a measurable map, by 
\begin{flalign*}
(f)_{U}:=\mint_{U}f(x) \ \dx =\frac{1}{\snr{U}}\int_{U}f(x) \ \dx
\end{flalign*}
we mean the integral average of $f$ over $U$. With $h\colon \Omega\to \mathbb{R}$, $U\subset \Omega$, and $\gamma \in (0,1]$ being a~given number, we shall denote
\begin{flalign*}
[h]_{0,\gamma;U}:=\sup_{\substack{x,y\in U,\\x\not =y}}\frac{\snr{h(x)-h(y)}}{\snr{x-y}^{\gamma}}, \qquad [h]_{0,\gamma}\equiv [h]_{0,\gamma;\Omega}.
\end{flalign*} 
{Recall that in~\eqref{con} we defined $\mathcal{K}_{\psi,g}(\Omega)$ with an obstacle $\psi$ and the boundary datum $g \in W^{1,H(\cdot)}(\Omega)$.  By $\mathcal{K}_{\psi}(\Omega)$ we denote $\mathcal{K}_{\psi,g}(\Omega)$ with $\psi\equiv g$.}

\subsection{Double phase energy}\label{dp}  Let us present the main properties of energy density $H(\cdot)$ given by~\eqref{H} under assumption \eqref{pq}. With abuse of notation, we shall keep on denoting $H(x,t)=t^{p}+a(x)t^{q}$ for $t\ge 0$, that is, when in \eqref{H}, $z$ is a non-negative number. For our purposes, it is enough to recall that, as a generalized Young function,  $x\mapsto H(x,\cdot)$ is $\alpha$-H\"older continuous, $t\mapsto H(\cdot,t)$  is strictly convex and belongs to $C^{1}([0,\infty))\cap C^{2}((0,\infty))$. By the Fenchel--Young conjugate of $H$, we mean the function $H^{*}(x,t):=\sup_{s\ge 0}\left\{st-H(x,s)\right\}$. Direct consequences of the definition of $H^{*}$ are the following equivalence
\begin{flalign}\label{ex3}
H^{*}\left(x, {H(x,t)}/{t}\right)\sim H(x,t) \quad\text{for all }\ (x,t)\in \Omega \times \mathbb{R}^{n},
\end{flalign}
holding up to constants depending only on $p$ and $q$,
and the Young inequalities 
\begin{flalign*}
\begin{cases}
\ st\le \varepsilon H^{*}(x,t)+\varepsilon^{1-q}H(x,s),\\
\ st \le \varepsilon H(x,s)+\varepsilon^{-\frac{1}{p-1}}H^{*}(x,t),
\end{cases}
\end{flalign*}
holding for all $x\in \Omega$, $s,t\in [0,\infty)$, and $\varepsilon\in(0,1)$, see \cite{bacomi,demi}. We shall often deal with the vector field
\begin{flalign*}
V_{t}(z):=\snr{z}^{\frac{t-2}{2}}z, \ \ t\in \{p,q\},
\end{flalign*}
which is of a common use to formulate the monotonicity properties of operators of the $p$-Laplacean type and related integral functionals. In this respect we record the following pointwise property
\begin{flalign*}0\le
\snr{V_{t}(z_{1})-V_{t}(z_{2})}^{2}\le c(n,t)\langle \snr{z_{1}}^{t-2}z_{1}-\snr{z_{2}}^{t-2}z_{2},z_{1}-z_{2}\rangle, \ \ t \in \{p,q\},
\end{flalign*}
for all $z_{1},z_{2}\in \mathbb{R}^{n}$. This, combined with $\eqref{A}_{2}$, renders that
\begin{flalign}\label{A-monotone}
0\le \mathcal{V}(z_{1},z_{2}):=&\snr{V_{p}(z_{1})-V_{p}(z_{2})}^{2}+a(x)\snr{V_{q}(z_{1})-V_{q}(z_{2})}^{2}\nonumber \\
\le &\,c\,\langle A(x,z_{1})-A(x,z_{2}),z_{1}-z_{2}\rangle,
\end{flalign}
with some $c= c(n,\nu,p,q)>0$ for all $x\in \Omega$ and any $z_{1},z_{2}\in \mathbb{R}^{n}$. Now, let us recall some common terminology in the framework of double phase functionals. For $B_{\rr}\subset \Omega$, we define
\begin{flalign*}
a_{i}(B_{\rr}):=\inf_{x\in B_{\rr}}a(x) \qquad \mbox{and} \qquad a_{s}(B_{\rr}):=\sup_{x\in B_{\rr}}a(x). 
\end{flalign*}
We will say that $a(\cdot)$ degenerates on $B_{\rr}$ if there holds
\begin{flalign}\label{adeg}
a_{i}(B_{\rr})\le 4[a]_{0,\alpha}\rr^{\alpha}.
\end{flalign}
The complementary condition reads as
\begin{flalign}\label{andeg}
a_{i}(B_{\rr})>4[a]_{0,\alpha}\rr^{\alpha}.
\end{flalign}
Those two conditions in some sense quantify the closeness of $a(\cdot)$ to {the set} of its zero points. Since in regime \eqref{adeg}, $a_{s}(B_{\rr})\le 6[a]_{0,\alpha}\rr^{\alpha}$, it follows that \eqref{adeg} is stable when $\rr$ increases. {On the other hand  under \eqref{andeg} we have $a_{s}(B_{\rr})\le 2a_{i}(B_{\rr})$, so} \eqref{andeg} is stable for shrinking balls. In accordance with this terminology, we also mention the auxiliary Young's functions
\begin{flalign*}
H^{-}_{B_{\rr}}(z):=\snr{z}^{p}+a_{i}(B_{\rr})\snr{z}^{q} \quad \mbox{and} \quad H^{+}_{B_{\rr}}(z):=\snr{z}^{p}+a_{s}(B_{\rr})\snr{z}^{q},
\end{flalign*}
which turn out to be useful in dealing with the regularity theory for double/multi phase functionals, see \cite{bacomi,comib,comi,demi,deoh}. {In the following, we will also consider the double phase integrand
\begin{flalign}\label{Hgamma}
{H_\sigma}(x,z):=\snr{z}^{p\sigma}+a(x)^{\sigma}\snr{z}^{q\sigma}, \ \ \frac{1}{p}<\sigma\le1.
\end{flalign}
Clearly $ {{H_\sigma}}(x,z)\sim (H(x,z))^{\sigma}$ for all $(x,z)\in \Omega\times \mathbb{R}^{n}$, in fact
\begin{flalign}\label{Hgamma-comp}
H^{\sigma}(x,z)\le {H_\sigma}(x,z)\le 2H^{\sigma}(x,z)  \ \ \mbox{for all} \ \ (x,z)\in \Omega\times \mathbb{R}^{n}.
\end{flalign}}

\subsection{Fuctional setting}\label{ssec:setting} There are various approaches how to describe general Musielak-Orlicz spaces, cf.~\cite{hahab,IC-pocket,IC-b}. We shall specialize them from the very beginning to those related to energy density $H(\cdot)$ given by~\eqref{H}. We define Musielak-Orlicz space 
\begin{flalign*}
L^{H(\cdot)}(\Omega):=\left\{w\in L^{1}_{\mathrm{loc}}(\Omega)\colon \int_{\Omega}H(x,w) \ \dx <\infty \right\},
\end{flalign*}
equipped with the Luxemburg norm
\begin{flalign*}
\nr{w}_{L^{H(\cdot)}(\Omega)}:=& 
\inf\left\{\lambda> 0\colon \int_{\Omega}H\left(x,\tfrac{1}{\lambda}w\right) \ \dx\le 1\right\}.
\end{flalign*}
{When we denote the modular 
$\rho_{H(\cdot)}(w):= \int_{\Omega}H\left(x,w\right)\dx,$
the structure of $H(\cdot)$ ensures that
\begin{flalign}\label{nr1}
\min\left\{\left(\rho_{H(\cdot)}(w)\right)^{\frac{1}{p}},\left(\rho_{H(\cdot)}(w)\right)^{\frac{1}{q}}\right\}\le \nr{w}_{L^{H(\cdot)}(\Omega)}\le\max \left\{\left(\rho_{H(\cdot)}(w)\right)^{\frac{1}{p}},\left(\rho_{H(\cdot)}(w)\right)^{\frac{1}{q}}\right\}
\end{flalign}
and
\begin{flalign}\label{nr2}
\min\left\{\left(\rho_{H^{*}(\cdot)}(w)\right)^{\frac{q}{q+1}},\right.&\left.\left(\rho_{H^{*}(\cdot)}(w)\right)^{\frac{p}{p+1}}\right\}\le \nr{w}_{L^{H^*(\cdot)}(\Omega)}\nonumber \\
\le&\max \left\{\left(\rho_{H^{*}(\cdot)}(w)\right)^{\frac{q}{q+1}},\left(\rho_{H^{*}(\cdot)}(w)\right)^{\frac{p}{p+1}}\right\}
\end{flalign}}
see e.g. \cite{hahab,mawo}. In particular, if $v\in L^{H(\cdot)}(\Omega)$ and $w\in L^{H^{*}(\cdot)}(\Omega)$, we have the H\"older inequality
{\begin{flalign*}
\left | \ \int_{\Omega}vw \ \dx \  \right |\le& 2\nr{v}_{L^{H(\cdot)}(\Omega)}\nr{w}_{L^{H^{*}(\cdot)}(\Omega)}\nonumber \\
\le& 2\max \left\{\left(\rho_{H(\cdot)}(w)\right)^{\frac{1}{p}},\left(\rho_{H(\cdot)}(w)\right)^{\frac{1}{q}}\right\} \max \left\{\left(\rho_{H^{*}(\cdot)}(w)\right)^{\frac{q}{q+1}},\left(\rho_{H^{*}(\cdot)}(w)\right)^{\frac{p}{p+1}}\right\},
\end{flalign*}}
where in the last inequality \eqref{nr1} and \eqref{nr2} are employed. Since the nonlinear tensor $A(\cdot)$ satisfies \eqref{A}, problem \eqref{A0} is naturally set in the Musielak-Orlicz-Sobolev space
\begin{flalign*}
W^{1,H(\cdot)}(\Omega):=\left\{w\in W^{1,1}(\Omega)\colon 
\ w,|Dw|\in {L^{H(\cdot)}(\Omega)}\right\},
\end{flalign*}
equipped with the norm $ \nr{w}_{W^{1,H(\cdot)}(\Omega)}:=  \nr{w}_{L^{H(\cdot)}(\Omega)}+\nr{Dw}_{L^{H(\cdot)}(\Omega )}$. Upon such a definition $W^{1,H(\cdot)}(\Omega)$ is a Banach space, which, due to the properties of $H(\cdot)$, is separable and reflexive. The dual space can be characterized by the means of the Fenchel-Young conjugate of $H(\cdot)$, namely we have $(W^{1,H(\cdot)}(\Omega))^{*}\sim W^{1,H^{*}(\cdot)}(\Omega)$.  Space $W^{1,H(\cdot)}_{\mathrm{loc}}(\Omega)$ is defined in the standard way.  We shall also define zero-trace space $W^{1,H(\cdot)}_{0}(\Omega)$ as a closure in $W^{1,H(\cdot)}(\Omega)$ of $C^{\infty}_{c}(\Omega)$-functions. Justification of this choice of definition requires some comments, since it is known that in inhomogeneous spaces smooth functions can be not dense~\cite{eslemi,fomami,zh}. 
\begin{remark}\label{rem:density}\rm In general, to get density of regular functions (smooth/Lipschitz) in norm in Musielak-Orlicz-Sobolev spaces, besides the (doubling) type of growth of $H(\cdot)$, its speed of growth has to be balanced with the regularity in the spacial variable, see the general study in~\cite{yags}. In turn, the assumption ensuring that this definition of $W^{1,H(\cdot)}_{0}(\Omega)$ makes sense, is closeness condition~\eqref{pq} imposed on the powers. In fact, the natural topology for Musielak-Orlicz-Sobolev spaces is the modular one, i.e. the one coming from  the notion of modular convergence~\cite{yags,IC-b,hahab}. We say that a sequence $(w_j)_{j \in \N}\subset L^{H(\cdot)}(\Omega)$ converges to $w$ modularly in $L^{H(\cdot)}(\Omega)$ if
\begin{flalign*}
\lim_{j\to \infty}w_{j}(x)=w(x) \ \ \mbox{for a.e.} \ x\in \Omega \quad \mbox{and}\quad \lim_{j\to \infty}\int_{\Omega}H(x,w_{j}-w) \, \dx=0.
\end{flalign*}
Consequently, $w_j\to w$ modularly in $W^{1,H(\cdot)}(\Omega)$ if both $w_j\to w$ and $Dw_j\to Dw$ modularly in $L^{H(\cdot)}(\Omega)$. Since the growth of $H$ is doubling, the modular convergence is equivalent to the norm convergence~\cite{IC-b,hahab}. Condition~\eqref{pq} was introduced and proven to be sharp for modular density in $W^{1,H(\cdot)}_{0}(\Omega)$ in~\cite{eslemi}. It plays the role of the assumption on log--H\"older continuity of the variable exponent, cf.~\cite{bacomi-st}. \end{remark}
Let us state a density lemma in the form useful in our investigations.
\begin{lemma}\cite{eslemi}\label{l1}
Under assumption \eqref{pq}, for any given $w\in W^{1,H(\cdot)}(\Omega)\cap W^{1,1}_0(\Omega)$ there exists a sequence $(w_{j})_{j \in \mathbb{N}}\subset C^{\infty}_{c}(\Omega)$ so that $(w_{j})_{j\in \N}$ converges to $w$ almost everywhere and in $W^{1,p}(\Omega)$, whereas $|Dw_j- Dw|\to 0$ modularly in $L^{H(\cdot)}(\Omega)$.
\end{lemma} 
We recall the intrinsic Sobolev-Poincar\'e inequalities.
\begin{lemma}\label{sopo} Suppose $\Omega\subset \mathbb{R}^{n}$ is a bounded, open set, $B_\rr\Subset\Omega,$ $\rr\le 1$,  and  \eqref{pq} is in force. Then there exist $d_{1}=d_{1}(n,p,q)\in(0,1)$, $d_{2}=d_{2}(n,p,q)>1$ and a positive constant $c=c(n,p,q,[a]_{0,\alpha},\alpha,\|Dw\|_{L^p(\Omega)}),$ such that for every  $w \in W^{1,H(\cdot)}(\Omega)$ the following inequalities hold
\begin{flalign}\label{sopo1}
\mint_{B_\rr} H\left(x,\frac{w-(w)_\rr}{\rr}\right)\ \dx\le c 
\left(\mint_{B_\rr} H(x,Dw )^{d_{1}}\ \dx\right)^\frac{1}{d_{1}}
\end{flalign}
and 
\begin{flalign}\label{sopo2}
\left(\mint_{B_{\rr}}H\left(x,\frac{w-(w)_{\rr}}{\rr}\right)^{d_{2}} \ \dx\right)^{\frac{1}{d_{2}}}\le c\mint_{B_{\rr}}H(x,Dw) \ \dx.
\end{flalign}
The analogous of estimates \eqref{sopo1} and \eqref{sopo2} holds if instead of $w-(w)_{\rr}$ we consider any $w\in W^{1,H(\cdot)}(\Omega)$ with $\left.w\right |_{\partial B_{\rr}}=0$. Furthermore, if $w\in W^{1,H(\cdot)}_{0}(\Omega)$, \begin{flalign}\label{sopo3}
\int_{\Omega}H(x,w) \ \dx\le c\int_{\Omega}H(x,Dw) \ \dx,
\end{flalign}
for $c=c(n,p,q,[a]_{0,\alpha},\alpha,\nr{Dw}_{L^{p}(\Omega)},\diam(\Omega))$.
\end{lemma}
\begin{proof}
Inequality \eqref{sopo1} can be found in \cite[Theorem 2.13]{Ok}. We shall concentrate now on~\eqref{sopo2}. Fix any $d_{2}\in \left(1,\frac{np}{q(n-p)}\right]$ and notice that, by \eqref{pq}, this position makes sense and $qd_{2}\le p^{*}<q^{*}$. Let $B_{\rr}\Subset \Omega$ be any ball as in the assumptions and consider the degenerate scenario \eqref{adeg}. Then, from the standard Sobolev-Poincar\'e inequality and \eqref{pq} we have
\begin{flalign*}
\mint_{B_{\rr}}H\left(x,\frac{w-(w)_{\rr}}{\rr}\right)^{d_{2}} \ \dx\le &c\mint_{B_{\rr}}\left |\frac{w-(w)_{\rr}}{\rr} \right |^{pd_{2}}+a(x)^{d_{2}}\left |\frac{w-(w)_{\rr}}{\rr} \right |^{qd_{2}} \ \dx\nonumber\\
\le &c\left(\mint_{B_{\rr}}\left |\frac{w-(w)_{\rr}}{\rr} \right |^{p^{*}} \ \dx \right)^{\frac{pd_{2}}{p^{*}}}+c\rr^{\alpha d_{2}}\left(\mint_{B_{\rr}}\left|\frac{w-(w)_{\rr}}{\rr}\right |^{p^{*}} \ \dx\right)^{\frac{qd_{2}}{p^{*}}}\nonumber \\
\le&c\rr^{\alpha d_{2}}\left(\mint_{B_{\rr}}\snr{Dw}^{p} \ \dx\right)^{\frac{d_{2}(q-p)}{p}}\left(\mint_{B_{\rr}}\snr{Dw}^{p} \ \dx\right)^{d_{2}}\nonumber\\
&+c\left(\mint_{B_{\rr}}\snr{Dw}^{p} \ \dx\right)^{d_{2}}\le c\left(\mint_{B_{\rr}}\snr{Dw}^{p}\ \dx\right)^{d_{2}},
\end{flalign*}
with $c=c(n,p,q,[a]_{0,\alpha},\alpha,\nr{Dw}_{L^{p}(\Omega)})$. Now consider the opposite condition, i.e., \eqref{andeg}. In this case, it is easy to see that $a_{s}(B_{\rr})\le \tfrac{3}{2}a_{i}(B_{\rr})$, therefore
\begin{flalign*}
\mint_{B_{\rr}}H\left(x,\frac{w-(w)_{\rr}}{\rr}\right)^{d_{2}} \ \dx\le& c\mint_{B_{\rr}}\left |\frac{w-(w)_{\rr}}{\rr} \right |^{pd_{2}}+\left |\frac{a_{i}(B_{\rr})^{\frac{1}{q}}(w-(w)_{\rr})}{\rr} \right |^{qd_{2}} \ \dx\nonumber\\
\le &c\left(\mint_{B_{\rr}}\snr{Dw}^{p} \ \dx\right)^{d_{2}}+\left(\mint_{B_{\rr}}\left |\frac{a_{i}(B_{\rr})^{\frac{1}{q}}(w-(w)_{\rr})}{\rr} \right |^{q^{*}} \ \dx\right)^{\frac{qd_{2}}{q^{*}}}\nonumber \\
\le &c\left(\mint_{B_{\rr}}\snr{Dw}^{p} \ \dx\right)^{d_{2}}+c\left(\mint_{B_{\rr}}a_{i}(B_{\rr})\snr{Dw}^{q} \ \dx\right)^{d_{2}}\nonumber \\
\le &c\left(\mint_{B_{\rr}}H(x,Dw) \ \dx\right)^{d_{2}},
\end{flalign*}
for $c=c(n,p,q)$. In both cases we obtain \eqref{sopo2}. Finally, for \eqref{sopo3}, we recall that $\Omega$ is bounded, therefore we can find a ball $B_{R}\subset \mathbb{R}^{n}$ with $R:=10\diam(\Omega)$ such that $\Omega \Subset B_{R}$. Now, if $w\in W^{1,H(\cdot)}_{0}(\Omega)$ we can define $\tilde{w}$ as an extension of $w$ by zero outside $\Omega$. 
Applying 
\eqref{sopo1} on $B_{R}$, (with $w-(w)_{R}$ replaced by $\tilde{w}$, of course), and H\"older's inequality we have
\begin{flalign*}
\int_{\Omega}H(x,w) \ \dx \le &\max\{1,10\diam(\Omega)\}^{q}\snr{B_{R}}\mint_{B_{R}}H(x,\tilde{w}/R) \ \dx\nonumber 
\le c\snr{B_{R}} \mint_{B_{R}}H(x,\tilde{w}/R) \ \dx \nonumber \\
\le& c\snr{B_{R}}\left(\mint_{B_{R}}H(x,D\tilde{w})^{d_{1}} \ \dx\right)^{\frac{1}{d_{1}}}
\le c\int_{B_{R}}H(x,D\tilde{w}) \ \dx =c\int_{\Omega}H(x,Dw) \ \dx,
\end{flalign*}
with {\it $c=c(n,p,q,[a]_{0,\alpha},\alpha,\nr{Dw}_{L^{p}(\Omega)},\diam(\Omega))$.}
\end{proof}

\section{Intrinsic capacities and intrinsic Hausdorff measures}\label{haus} We define the \emph{intrinsic $H(\cdot)$-capacity} and recall its main features exactly in the form we need. For more details and generalizations, we refer the reader to \cite{bahaha}. Given a~compact set $K\subset \Omega$, we denote its relative $H(\cdot)$-capacity as
\begin{flalign*}
cap_{H(\cdot)}(K,\Omega):=\inf_{f\in \mathcal{R}(K)}\int_{\Omega}H(x,Df) \ \dx,
\end{flalign*}
where the set of test functions is
\begin{flalign*}
\mathcal{R}_{H(\cdot)}(K):=\left\{ f\in W^{1,H(\cdot)}(\Omega)\cap C_{0}(\Omega)\colon f\ge 1 \ \ \mbox{in} \ \ K \right\}.
\end{flalign*}
As usual, for open subsets $U\subset \Omega$ and general $E\subset \Omega$ we have
\begin{flalign*}
cap_{H(\cdot)}(U,\Omega):=\sup_{{\substack{K\subset U,\\ K \ \mbox{compact}}}}cap_{H(\cdot)}(K,\Omega)
\end{flalign*}
and then
\begin{flalign*}
cap_{H(\cdot)}(E,\Omega):=\inf_{{\substack{E\subset U\subset \Omega, \\ U \ \mbox{open}}}}cap_{H(\cdot)}(U,\Omega).
\end{flalign*}
The structure of $H(\cdot)$ and \eqref{pq} guarantee that $cap_{H(\cdot)}$ enjoys the standard properties of~Sobolev capacities. In particular, $cap_{H(\cdot)}$ is Choquet, which means that
\begin{flalign}\label{cho}
cap_{H(\cdot)}(E)=\sup\left\{cap_{H(\cdot)}(K)\colon K\subset E \ \mbox{is a compact set}\right\}.
\end{flalign}
Moreover, again from \eqref{pq} and the convexity of $z\mapsto H(\cdot, z)$, we see that the relative capacity $cap_{H(\cdot)}$ is equivalent to the capacity $C_{H(\cdot)}$ defined in \cite[Section~3]{bahaha}, see \cite[Theorem~7.3 and Proposition~7.5]{bahaha}.
\begin{remark}\label{r3}
\emph{Given a compact $K\Subset \Omega$, when working with any function $f\in \mathcal{R}_{H(\cdot)}(K)$, there is no loss of generality in assuming $0\le f\le 1$ on $\Omega$. {Since $f\in C_{0}(\Omega)$, $f\ge 1$ on $K$ and the map $t\mapsto \min\{t,1\}$ is Lipschitz,} it follows that $\tilde{f}:=\min\{f,1\} \in \mathcal{R}_{H(\cdot)}(K)$. Moreover,
\begin{flalign*}
\int_{\Omega}H(x,D\tilde{f}) \ \dx =&\int_{\{x\in \Omega\colon f(x)< 1\}}H(x,Df) \ \dx \le \int_{\Omega}H(x,Df) \ \dx,
\end{flalign*}
so our claim follows from the definition of $cap_{H(\cdot)}(K)$. Moreover, Lemma \ref{l1} guarantees the density of smooth and compactly supported functions in $W^{1,H(\cdot)}_{0}(\Omega)$, thus, when dealing with functions $f\in\mathcal{R}_{H(\cdot)}(K)$ there is no loss of generality in assuming $f\in C^{\infty}_{c}(\Omega)$.}
\end{remark}

Naturally associated to those capacities is the concept of \emph{intristic Hausdorff measures}, introduced in \cite{demi}, see also \cite{ni,tu}. For any $n$-dimensional open ball $B\subset \Omega$  of radius $\rr(B)\in (0,\infty)$, we define
\begin{flalign*}
h_{H(\cdot)}(B):=\int_{B}H\left(x,\tfrac{1}{\rr(B)}\right) \ \dx.
\end{flalign*}Note that there is no difference in the following in taking closed balls in this respect. Applying the standard Carath\'eodory's construction we obtain an outer measure, in fact, let $E\subset \Omega$ be any subset. We define the $\delta$-approximating Hausdorff measure of $E$, $\mathcal{H}_{H(\cdot),\delta}(E)$ with $\delta \leq 1$, by
\begin{flalign*}
\mathcal{H}_{H(\cdot),\delta}(E)=\inf_{\mathcal{C}_{E}^{\delta}}\sum_{j}h_{H(\cdot)}(B_{j}),
\end{flalign*}
where
\begin{flalign*}
\mathcal{C}_{E}^{\delta}=\left \{\ \{B_{j}\}_{j\in \N}\ \mathrm{is \ a \ countable \ collection \ of \ balls \ B_j \subset \Omega \ covering \ }E,\, \ \rr(B_{j})\le \delta \  \right\}.
\end{flalign*}
As $0<\delta_{1}<\delta_{2}<\infty$ implies $\mathcal{C}_{E}^{\delta_{1}}\subset \mathcal{C}_{E}^{\delta_{2}}$, we have that $\mathcal{H}_{H(\cdot), \delta_1}(E)\ge \mathcal{H}_{H(\cdot),\delta_{2}}(E)$ and there exists the limit
\begin{flalign*}
\mathcal{H}_{H(\cdot)}(E):=\lim_{\delta \to 0}\mathcal{H}_{H(\cdot),\delta}(E)=\sup_{\delta>0}\mathcal{H}_{H(\cdot),\delta}(E)\;.
\end{flalign*}
By standard arguments, $\mathcal{H}_{H(\cdot)}$ is a Borel regular measure. The intrinsic Hausdorff measure related to $ {H}_{\sigma}(\cdot)$ given by~\eqref{Hgamma} is denoted of course as $\mathcal{H}_{H_{\sigma}(\cdot)}$.  Notice that, for $\sigma \in (1/p,1]$, function ${H}_\sigma(\cdot)$ is a generalized Young function with the same structure of $H(\cdot)$. Therefore, $\mathcal{H}_{H_{\sigma}(\cdot)}$  has the same properties as $\mathcal{H}_{H(\cdot)}$. Moreover, due to~\eqref{Hgamma-comp} $\mathcal{H}_{H_{\sigma}(\cdot)}\sim  \mathcal{H}_{H^{\sigma}(\cdot)}$. Let us recall a result which relates intrinsic Hausdorff measures with the corresponding intrinsic capacity.
\begin{proposition}\cite{demi}\label{eq}
For $H(\cdot)$ given by \eqref{H} under assumption \eqref{pq}, if $E\subset \mathbb{R}^{n}$ is such that $\mathcal{H}_{H(\cdot)}(E)<\infty$, then $cap_{H(\cdot)}(E)=0$.
\end{proposition}
We conclude this section with the proof of a first result on removable sets for solutions of \eqref{A0}. Following \cite[Chapter 2]{hekima} and taking into account  Remark~\ref{rem:density}, for $E\subset \Omega$ relatively closed, we say that
\begin{flalign*}
W^{1,H(\cdot)}_{0}(\Omega)=W^{1,H(\cdot)}_{0}(\Omega\setminus E)
\end{flalign*}
if for any given $w\in W^{1,H(\cdot)}_{0}(\Omega)$ there exists a sequence $(w_{j})_{j\in \N}\subset C^{\infty}_{c}(\Omega\setminus E)$ such that  $w_j\to w$ modularly in $W^{1,H(\cdot)}_{0}(\Omega)$. Now we are ready to state our first two results, which clarify when a set is negligible in $W^{1,H(\cdot)}(\Omega)$. The already classical version of this fact stated in the Sobolev space $W^{1,p}$ can be found in \cite[Section 2.42]{hekima}. 
\begin{lemma}\label{l2}
Suppose that $E$ is a relatively closed subset of $\Omega$. Then 
\begin{flalign*}
W^{1,H(\cdot)}_{0}(\Omega)=W^{1,H(\cdot)}_0(\Omega\setminus E)\ \ \mbox{if and only if} \ \ cap_{H(\cdot)}(E)=0.
\end{flalign*}
\end{lemma}
\begin{proof}{ According to Remark~\ref{rem:density}}, we can work with the modular convergence. Assume first that $cap_{H(\cdot)}(E)=0$. The inclusion $W^{1,H(\cdot)}_0(\Omega\setminus E)\subset W^{1,H(\cdot)}_0(\Omega)$ is trivial, so we only need to show that $W^{1,H(\cdot)}_0(\Omega)\subset W^{1,H(\cdot)}_0(\Omega\setminus E)$. Let $\varphi \in C^{\infty}_{c}(\Omega)$. Since $cap_{H(\cdot)}(E)=0$, according to Remark \ref{r3}, there exists a sequence $(f_{j})_{j \in \N}\subset \left(\mathcal{R}_{H(\cdot)}(E)\cap C^{\infty}_{c}(\Omega)\right)$ such that
\begin{flalign}\label{fj}
0\le f_{j}\le 1 \ \ \mbox{and} \ \ \lim_{j\to \infty}\int_{\Omega}H(x,Df_{j}) \ \dx = 0
\end{flalign}
{and, for any $j\in \N$, the} map $\psi_{j}:=(1-f_{j})\varphi$ has support contained in $\Omega\setminus E${. Then} $(\psi_{j})_{j \in \N}\subset C^{\infty}_{c}(\Omega\setminus E)$. Moreover, \eqref{fj} and the dominated convergence theorem allow concluding that
\begin{flalign*}
\int_{\Omega\setminus E}H(x,D\psi_{j}-D\varphi) \ \dx=0,
\end{flalign*}
therefore $\varphi\in W^{1,H(\cdot)}_{0}(\Omega\setminus E)$. Since by Lemma \ref{l1} and the dominated convergence theorem we can approximate any $w \in W^{1,H(\cdot)}_{0}(\Omega \setminus E)$ in the modular {topology} 
via the sequence of truncations $({w}_{k})_{k\in \mathbb{N}}:=(\max \{-k,\min\{w,k\} \})_{k \in \mathbb{N}}$ we have
\begin{flalign*}
W^{1,H(\cdot)}_{0}(\Omega)\subset W^{1,H(\cdot)}_{0}(\Omega\setminus E),
\end{flalign*}
and the `if' part of the lemma is done. For the `only if' part, let $K\subset E$ be {arbitrary compact set}. By \eqref{cho}, it is sufficient to show that $cap_{H(\cdot)}(K)=0$. Let us fix an arbitrary $f \in \mathcal{R}_{H(\cdot)}(K)$. Since $W^{1,H(\cdot)}_{0}(\Omega)=W^{1,H(\cdot)}_{0}(\Omega\setminus E)$, there exists a sequence $(\varphi_{j})_{j \in \N}\subset C^{\infty}_{c}(\Omega\setminus E)$ such that {$\varphi_j\to f$ a.e. in $\Omega$ and} $\lim_{j\to \infty}\int_{\Omega}H(x,D\varphi_{j}-Df) \ \dx=0,$ thus the map $g_{j}:=f-\varphi_{j}\in \mathcal{R}_{H(\cdot)}(K)$ for all $j\in \N$. As a consequence of the definition of $cap_{H(\cdot)}$, we have
\begin{flalign*}
cap_{H(\cdot)}(K)\le \lim_{j\to \infty}\int_{\Omega}H(x,Dg_{j}) \ \dx=0,
\end{flalign*}
and the lemma follows.
\end{proof}
As a direct consequence of Lemma \ref{l2}, we show that sets of zero $\mathcal{H}_{H(\cdot)}$-measure are removable for $\mathcal{A}_{H(\cdot)}$-harmonic maps.
\begin{corollary}\label{coro:H0}
Let $E\subset \Omega$ be a relatively closed subset of $\Omega$ such that $\mathcal{H}_{H(\cdot)}(E)=0$ and $u\in W^{1,H(\cdot)}(\Omega)$ satisfying
\begin{flalign}\label{70}
\int_{\Omega\setminus E}A(x,Du)\cdot Dw \ \dx=0 
\end{flalign}
for all $w\in W^{1,H(\cdot)}_{0}(\Omega\setminus E)$. Then, $u$ is a solution to \eqref{A0} on the whole $\Omega$.
\end{corollary}
\begin{proof}
Since $\mathcal{H}_{H(\cdot)}(E)=0$, by Proposition \ref{eq} we have that $cap_{H(\cdot)}(E)=0$, thus by Lemma~\ref{l2} we can conclude that $W^{1,H(\cdot)}_{0}(\Omega\setminus E)=W^{1,H(\cdot)}_{0}(\Omega)$, so \eqref{70} actually holds for all $w \in W^{1,H(\cdot)}_{0}(\Omega)$.
\end{proof}
\vspace{2mm}
\section{The obstacle problem for the double phase energy}\label{o}
This section is devoted to the study on the existence and main regularity properties of solutions to the obstacle problem defined by means of a differential operator with structure~\eqref{A}.
  
{\subsection{Existence and uniqueness of solutions to the obstacle problem}
In this section we infer existence and uniqueness of solutions to the obstacle problem as a consequence of classical results on solvability in reflexive Banach spaces and comparison principles. We assume $\mathcal{K}_{\psi,g}(\Omega)\not =\emptyset$, see~\eqref{con}. Notice that when $\psi\equiv g$, we have $\mathcal{K}_{\psi,g}(\Omega)\not =\emptyset$ since $\psi\in \mathcal{K}_{\psi,g}(\Omega)$. Our result reads as follows.

\begin{proposition}\label{prop:ex-uni}
Suppose that assumptions \eqref{A} and \eqref{pq} are satisfied and $\psi,g\in W^{1,H(\cdot)}(\Omega)$ are such that $\mathcal{K}_{\psi,g}(\Omega)\not =\emptyset$. Then there exists a~unique weak solution $v\in\mathcal{K}_{\psi,g}(\Omega)$ to problem \eqref{obs}.
\end{proposition}

We recall some elementary facts about monotone operators defined on a reflexive Banach space, which finally will be applied to the operator $\mathcal{A}_{H(\cdot)}$, defined in~\eqref{calAH}.
  
\begin{definition}\label{d1}
Let $X$ be a reflexive Banach space with dual $X^{*}$ and $\langle \cdot,\cdot\rangle$ denote a pairing between $X^{*}$ and $X$. If $K\subset X$ is any closed, convex subset, then a map $T\colon K\to X^{*}$ is called monotone if it satisfies $
\langle Tw-Tv,w-v \rangle\ge 0 \ \ \mbox{for all} \ \ w,v\in K.
$ Moreover, we say that $T$ is coercive if there exists a $w_{0}\in K$ such that
\begin{flalign*}
\lim_{\nr{w}_{X}\to \infty}\frac{\langle Tw,w-w_{0} \rangle}{\nr{w}_{X}}=\infty \ \ \mbox{for all} \ \ w\in K.
\end{flalign*}
\end{definition}
The following proposition guarantees the existence of solution to variational inequalities associated to monotone operators.
\begin{proposition}\cite{kiod} \label{exist}
Let $K\subset X$ be a nonempty, closed, convex subset and $T\colon K\to X^{*}$ be monotone, weakly continuous and coercive on $K$. Then there exists an element $v\in K$ such that $\langle Tv,w-v\rangle\ge 0$ for all $w\in K$. 
\end{proposition}
\begin{proof}[Proof of Proposition~\ref{prop:ex-uni}] We verify all the assumptions of Proposition~\ref{exist} step by step. 

\medskip

\noindent \emph{Step 1: the space setting}. Let us start by noticing that $\mathcal{A}_{H(\cdot)}$ is defined on $X=W^{1,H(\cdot)}(\Omega)$, which  is a reflexive Banach space due to the structure of $H(\cdot)$. Notice that $\mathcal{A}_{H(\cdot)}(W^{1,H(\cdot)}(\Omega))\subset (W^{1,H(\cdot)}(\Omega))^{*}$. Indeed, when $v\in W^{1,H(\cdot)}(\Omega)$ and $w\in C_0^\infty(\Omega)$ from {Young's inequality,~\eqref{ex3}, and Lemma~\ref{sopo} we get
\begin{flalign}\label{i1}
\snr{\langle \mathcal{A}_{H(\cdot)}v,w \rangle}\le &L\int_{\Omega}\frac{H(x,Dv)}{\snr{Dv}}\snr{Dw} \ \dx \le c\left \| \frac{H(\cdot,Dv)}{\snr{Dv}}\right \|_{L^{H^{*}(\cdot)}(\Omega)}\nr{Dw}_{L^{H(\cdot)}(\Omega)}\nonumber \\
\le &c\max\left\{\left(\int_{\Omega}H(x,Dv) \ \dx\right)^{\frac{q}{q+1}},\left(\int_{\Omega}H(x,Dv) \ \dx\right)^{\frac{p}{p+1}}\right\}\nr{Dw}_{L^{H(\cdot)}(\Omega)}\nonumber \\
\le &c\nr{w}_{W^{1,H(\cdot)}(\Omega)}
\end{flalign}}
with $c=c(L,p,q,\nr{H(\cdot,Dv)}_{L^{1}(\Omega)})$, thus $\mathcal{A}_{H(\cdot)}v\in (W^{1,H(\cdot)}(\Omega))^{*}$.

\medskip

\noindent \emph{Step 2: the set $\mathcal{K}_{\psi,g}(\Omega)\subset W^{1,H(\cdot)}(\Omega)$  is closed and convex.}   In fact, for $\lambda \in [0,1]$ and $w_{1},w_{2}\in \mathcal{K}_{\psi,g}(\Omega)$, define $w_{\lambda}:=\lambda w_{1}+(1-\lambda)w_{2}$. Then $w_{\lambda}\in W^{1,H(\cdot)}(\Omega)$, $w_{\lambda}-g\in W^{1,H(\cdot)}_{0}(\Omega)$ and $w_{\lambda}\ge \psi$ a.e. in $\Omega$. Moreover, if $w\in W^{1,H(\cdot)}(\Omega)$ and $(w_{j})_{j \in \N}\subset \mathcal{K}_{\psi,g}(\Omega)$ is any sequence such that $\lim_{j\to \infty}\int_{\Omega}H(x,Dw_{j}-Dw) \ \dx=0$, then by the continuity of the trace operator, $w-g\in W^{1,H(\cdot)}_{0}(\Omega)$ and, by Lebesgue's dominated convergence theorem, $w\ge \psi$ a.e. in $\Omega$. 

\medskip
\noindent \emph{Step 3: weak continuity}. Let $v,(v_{j})_{j \in \N}\subset W^{1,H(\cdot)}(\Omega)$ be such that $v_{j}\to v \ \ \mbox{in} \ \ W^{1,H(\cdot)}(\Omega)$. Then it follows that there exists $M=M(\nr{v}_{W^{1,H(\cdot)}(\Omega)})$ such that $\sup_{j\in \N}\nr{v_{j}}_{W^{1,H(\cdot)}(\Omega)}\le M$. Moreover, up to a subsequence, $\lim_{j\to \infty}v_{j}(x)=v(x)$ and $\lim_{j\to \infty}Dv_{j}(x)= Dv(x)$ for a.e. $x\in \Omega$. Since $z\mapsto A(\cdot,z)$ is continuous we have, for any $w\in W^{1,H(\cdot)}(\Omega)$, 
\begin{flalign}\label{67}
\lim_{j\to \infty}A(x,Dv_{j}(x))\cdot Dw(x)=A(x,Dv(x))\cdot Dw(x) \ \ \mbox{for a.e.} \ \ x\in \Omega.
\end{flalign}
Moreover, if $E\subset \Omega$ is any measurable subset, then  {as in~\eqref{i1} we have}
\begin{flalign}\label{68}
\left | \  \int_{E}A(x,Dv_{j})\cdot Dw \ \dx \ \right |\le &\int_{E}\snr{A(x,Dv_{j})\cdot Dw} \ \dx\nonumber\\
\le& L\int_{E}\frac{H(x,Dv_j)}{\snr{Dv_j}}\snr{Dw} \ \dx \le c\nr{Dw}_{L^{H(\cdot)}(E)},
\end{flalign}
for $c=c(L,p,q,\nr{v}_{W^{1,H(\cdot)}(\Omega)})$, by recalling the dependency of $M$. By \eqref{67} and \eqref{68}, we {can apply Vitali's convergence theorem} getting
\begin{flalign*}
\lim_{j\to \infty} \langle \mathcal{A}_{H(\cdot)}v_{j},w \rangle=\lim_{j\to \infty}\int_{\Omega}A(x,Dv_{j})\cdot Dw \ \dx =\int_{\Omega}A(x,Dv)\cdot Dw \ \dx=\langle \mathcal{A}_{H(\cdot)}v,w \rangle,
\end{flalign*}
therefore $\mathcal{A}_{H(\cdot)}$ is weakly continuous on $W^{1,H(\cdot)}(\Omega)$.
\medskip

\noindent \emph{Step 4: monotonicity and coercivity of $\mathcal{A}_{H(\cdot)}$}. Monotonicity results directly from~\eqref{A-monotone}. As for coercivity, we fix $w,w_{0}\in \mathcal{K}_{\psi,g}(\Omega)$ and, using $\eqref{A}_{2}$ and the weighted H\"older's and Young's inequalities we obtain
\begin{flalign*}
\langle\mathcal{A}_{H(\cdot)}w,w-w_{0} \rangle=&\int_{\Omega}A(x,Dw)\cdot(Dw-Dw_{0}) \ \dx\nonumber \\
\ge &\nu\int_{\Omega}H(x,Dw) \ \dx -c\left(\int_{\Omega}\snr{Dw}^{p} \ \dx\right)^{\frac{p-1}{p}}\left(\int_{\Omega}\snr{Dw_{0}}^{p} \ \dx\right)^{\frac{1}{p}}\nonumber \\
&-c\left(\int_{\Omega}a(x)\snr{Dw}^{q} \ \dx\right)^{\frac{q-1}{q}}\left(\int_{\Omega}a(x)\snr{Dw_{0}}^{q} \ \dx\right)^{\frac{1}{q}}\nonumber \\
\ge &\frac{\nu}{2}\int_{\Omega}H(x,Dw) \ \dx-\tilde{c},
\end{flalign*}
with $\tilde{c}=\tilde{c}(\nu,L,p,q,\nr{H(\cdot,Dw_{0})}_{L^{1}(\Omega)})$. From \eqref{nr1} we have
\begin{flalign*}
\int_{\Omega}H(x,Dw) \ \dx\ge \min\left\{\nr{Dw}_{L^{H(\cdot)}(\Omega)}^{p},\nr{Dw}_{L^{H(\cdot)}(\Omega)}^{q}\right\},
\end{flalign*}
therefore, merging the content of the two previous displays we obtain
\begin{flalign}\label{66}
\frac{\langle \mathcal{A}_{H(\cdot)}w,w-w_{0}\rangle}{\nr{Dw}_{L^{H(\cdot)}(\Omega)}}\ge&\frac{\nu}{2}\min\left\{\nr{Dw}_{L^{H(\cdot)}(\Omega)}^{p-1},\nr{Dw}_{L^{H(\cdot)}(\Omega)}^{q-1}\right\}-\frac{\tilde{c}}{\nr{Dw}_{L^{H(\cdot)}(\Omega)}}\to\infty
\end{flalign}
as $\nr{Dw}_{L^{H(\cdot)}(\Omega)}\to \infty$. This is enough for the coercivity condition, see Remark \ref{R5} below.

\medskip

\noindent \emph{Step 5: Comparison principle}. We show that if $v\in W^{1,H(\cdot)}(\Omega)$ is a solution to problem \eqref{obs}, $\tilde{v}\in W^{1,H(\cdot)}(\Omega)$ is a supersolution to \eqref{A0} and $w:=\min\{v,\tilde{v}\}\in \mathcal{K}_{\psi,g}(\Omega)$, then $\tilde{v}(x)\ge v(x)$ for a.e. $x\in \Omega$. Being $\tilde{v}$ a supersolution to \eqref{A0}, the map $\tilde{w}:=v-\min\{\tilde{v},v\}$ is an admissible test in \eqref{sux} and, being $v$ a solution to \eqref{obs} and $w\in \mathcal{K}_{\psi,g}(\Omega)$ we have
\begin{flalign*}
\begin{cases}
\ \int_{\Omega\cap \{x\colon \tilde{v}(x)< v(x) \}}A(x,D\tilde{v})\cdot (Dv-D\tilde{v}) \ \dx \ge 0,\\
\ \int_{\Omega\cap \{x\colon \tilde{v}(x)< v(x)\}}A(x,Dv)\cdot (D\tilde{v}-Dv) \ \dx\ge 0.
\end{cases}
\end{flalign*}
Adding the two inequalities in the above display and {using \eqref{A-monotone}}, we obtain
\begin{flalign*}
0\le \int_{\Omega\cap \{x\colon \tilde{v}(x)< v(x)\}}\left(A(x,D\tilde{v})-A(x,Dv)\right)\cdot (Dv-D\tilde{v}) \ \dx \le 0,
\end{flalign*}
thus either $\snr{\Omega\cap \{x\colon \tilde{v}(x)< v(x)\}}=0$ or $D\tilde{v}=Dv$ a.e. on $\Omega\cap \{x\colon \tilde{v}(x)< v(x)\}$. This second alternative is excluded by the fact that $w\in \mathcal{K}_{\psi,g}(\Omega)$, so $v-\tilde{v}\in W^{1,H(\cdot)}_{0}(\Omega\cap \{x\colon \tilde{v}(x)< v(x)\})$. Therefore $\snr{\Omega\cap \{x\colon \tilde{v}(x)< v(x)\}}=0$ and $\tilde{v}\ge v$ a.e. in $\Omega$.

\medskip

\noindent \emph{Step 6: Conclusion}. 
By {\emph{Steps 1-4}} and Remark \ref{R5} we see that $\mathcal{K}_{\psi,g}(\Omega)$ and $\mathcal{A}_{H(\cdot)}$ satisfy the assumptions of Proposition \ref{exist}, therefore there exists a solution to problem \eqref{obs}. Let us prove that it is unique. If there were two solutions $v_{1},v_{2}\in \mathcal{K}_{\psi,g}(\Omega)$, then, by recalling that each of those two solutions is an admissible competitor for the other and {using \eqref{A-monotone}} we obtain
\begin{flalign*}
0\le \int_{\Omega}\left(A(x,Dv_{1})-A(x,Dv_{2})\right)\cdot (Dv_{2}-Dv_{1}) \ \dx \le 0.
\end{flalign*}
Hence $Dv_{1}(x)=Dv_{2}(x)$ for a.e. $x\in \Omega$ and since $v_{1}-v_{2}\in W^{1,H(\cdot)}_{0}(\Omega)$, we can conclude that $v_{1}=v_{2}$ almost everywhere.
\end{proof}

\begin{remark}\label{R5}
\rm In the proof of Proposition \ref{prop:ex-uni}, \emph{Step 4} deserves some clarification. As to show coercivity, we cannot follow the usual path involving Poincar\'e's inequality, see e.g. \cite{chaly,fush}, given that the constant in \eqref{sopo3} depends in an increasing way on the $L^{p}$-norm of the gradient. This can be an obstruction for small values of $p$. Even though Definition \ref{d1} prescribes that 
\begin{flalign*}
\lim_{\nr{w}_{W^{1,H(\cdot)}(\Omega)}\to \infty}\frac{\langle\mathcal{A}_{H(\cdot)}w,w-w_{0}\rangle}{\nr{w}_{W^{1,H(\cdot)}(\Omega)}}=\infty \ \ \mbox{for all} \ \ w\in \mathcal{K}_{\psi,g}(\Omega),
\end{flalign*}
while we only have \eqref{66}, we still can prove existence. In fact, for $\Lambda \ge 0$, set
\begin{flalign*}
\mathcal{K}^{\Lambda}_{\psi,g}(\Omega):=\mathcal{K}_{\psi,g}(\Omega)\cap \left\{w\in W^{1,H(\cdot)}(\Omega)\colon \nr{Dw}_{L^{H(\cdot)}(\Omega)}\le \Lambda\right\}
\end{flalign*}
Then, by \eqref{sopo3} and \eqref{nr1} it easily follows that
\begin{flalign*}
\mathcal{K}^{\Lambda}_{\psi,g}(\Omega)\subset \left\{w\in W^{1,H(\cdot)}(\Omega)\colon \nr{w}_{W^{1,H(\cdot)}(\Omega)}\le \bar{c}=\bar{c}(n,p,q,\Lambda,\nr{g}_{W^{1,H(\cdot)}(\Omega)},\diam(\Omega))\right\},
\end{flalign*}
thus $\mathcal{K}_{\psi,g}^{\Lambda}(\Omega)$ is bounded (and, of course, closed and convex) in $W^{1,H(\cdot)}(\Omega)$. This is enough for our purposes, see \cite[Section 1.6]{kiod}.\end{remark}
Furthermore, we have the following direct consequence of comparison principle of \emph{Step~5} above. 
\begin{remark}\label{rem-i} If $u\in W^{1,H(\cdot)}(\Omega)$ is a solution to \eqref{A0}, it is 
a supersolution to the same equation. Thus, whenever $v\in \mathcal{K}_{u}(\Omega)$ is a solution to problem \eqref{obs}, then $u(x)\ge v(x)$ for a.e. $x\in \Omega$.
\end{remark}
}
\subsection{Regularity for the obstacle problem} 
Let us concentrate first on Gehring's and De Giorgi's theories.
\begin{proposition}\label{prop:GDG} Suppose that assumptions \eqref{pq} and \eqref{A} are satisfied, $\psi,g\in W^{1,H(\cdot)}(\Omega)$ are such that $\mathcal{K}_{\psi,g}(\Omega)\not =\emptyset$, and $v\in \mathcal{K}_{\psi,g}(\Omega)$ is a solution to the obstacle problem \eqref{obs}. Then we have~\eqref{in:ge} and~\eqref{v-bounded}.
\end{proposition}
\begin{proof}[Gehring's theory.] For $0<\rr\le 1$, let $B_{\rr}\Subset \Omega$, and take a cut-off function $\eta\in C^{1}_{c}(B_{\rr})$ such that $\chi_{B_{\rr/2}}\le \eta\le \chi_{B_{\rr}}$ and $\snr{D\eta}\le \rr^{-1}$. It is easy to see that the map $w:=v-\eta^{q}(v-\psi)\in \mathcal{K}_{\psi,g}(\Omega)$, so, by $\eqref{A}_{1,2}$, Young's inequality and \eqref{sopo1}, we have
\begin{flalign*}
\mint_{B_{\rr/2}}H(x,Dv) \ \dx\le& c\mint_{B_{\rr}}H\left(x,\frac{v-\psi}{\rr}\right) \ \dx+c\mint_{B_{\rr}}H(x,D\psi) \ \dx \nonumber \\
\le &c\left(\mint_{B_{\rr}}H(x,Dv)^{d_{1}} \ \dx\right)^{\frac{1}{d_{1}}}+c\mint_{B_{\rr}}H(x,D\psi) \ \dx,
\end{flalign*}
with $c=c(\textit{\data},\nr{Dv}_{L^{p}(\Omega)},\nr{D\psi}_{L^{p}(\Omega)})$. Now, we are in position to apply Gehring-Giaquinta-Modica's Lemma, \cite[Chapter 6]{gi}, to conclude that there exists a $\delta_{0}\in (0,\delta_{1})$ such that for all $\delta \in (0,\delta_{0})$ there holds
\begin{flalign*}
\left(\mint_{B_{\rr/2}}H(x,Dv)^{1+\delta} \ \dx\right)^{\frac{1}{1+\delta}}\le c\left(\mint_{B_{\rr}}H(x,D\psi)^{1+\delta} \ \dx\right)^{\frac{1}{1+\delta}}+c\mint_{B_{\rr}}H(x,Dv) \ \dx.
\end{flalign*}
Here $c$ and $\delta_{0}$ depend on $(\textit{\data},\nr{Dv}_{L^{p}(\Omega)},\nr{D\psi}_{L^{p}(\Omega)})$. Moreover, after a standard covering argument, we can conclude that $H(\cdot,Dv)^{1+\delta}\in L^{1}_{\mathrm{loc}}(\Omega)$ for all $\delta\in (0,\delta_{0})$.
\end{proof}
\begin{proof}[De Giorgi's Theory] Let $v$ be the solution to problem \eqref{obs} and $M$ be a non-negative constant to be adjusted in a few lines. For $s>0$ and $\lambda \in \mathbb{R}$, let us define
\begin{flalign*}
J^{\pm}_{v-M}(\lambda,s):=B_{s}\cap\left\{x\in \Omega\colon v(x)-M\gtrless\lambda\right\}.
\end{flalign*}
Fix radii $0<\rr<r\le 1$ and a ball $B_{r}\Subset \Omega$ and pick a cut-off function $\eta\in C^{1}_{c}(B_{r})$ so that $\chi_{B_{\rr}}\le \eta\le \chi_{B_{r}}$ and $\snr{D\eta}\le (r-\rr)^{-1}$. For $\kappa\ge 0$ and $M\ge \sup_{x\in B_{r}}\psi(x)$, the map $w^{+}:=v-\eta^{q}(v-M-\kappa)_{+}$ belongs by construction to $\mathcal{K}_{\psi,g}(\Omega)$, so testing \eqref{obs} against $w^{+}$, using $\eqref{A}_{1,2}$, Young's inequality and reabsorbing terms, we obtain
\begin{flalign}\label{28}
\int_{J^{+}_{v-M}(\kappa,\rr)}H\left(x,D(v-M-\kappa)_{+}\right) \ \dx \le c\int_{J^{+}_{v-M}(\kappa,r)}H\left(x,\frac{(v-M-\kappa)_{+}}{r-\rr}\right) \ \dx,
\end{flalign}
with $c=c(\nu,L,p,q)$. In a similar way, this time for any $\kappa\in \mathbb{R}$, we consider the map $w^{-}:=v+\eta^{q}(v-M-\kappa)_{-}$ where $\eta$ and $M$ are as before. Again, $w^{-}\in \mathcal{K}_{\psi,g}(\Omega)$, so inserting it as a test function in \eqref{obs} we get
\begin{flalign}\label{29}
\int_{J^{-}_{v-M}(\kappa,\rr)}H\left(x,D(v-M-\kappa)_{-}\right) \ \dx \le c\int_{J^{-}_{v-M}(\kappa,r)}H\left(x,\frac{(v-M-\kappa)_{-}}{r-\rr}\right) \ \dx,
\end{flalign}
where $c=c(\nu,L,p,q)$. Now we are in position to improve estimates \eqref{28}-\eqref{29} to
\begin{flalign}\label{dg2}
\int_{J_{v-M}^{\pm}(\kappa,\rr)}H\left(x,\frac{(v-M-\kappa)_{\pm}}{r}\right) \ \dx\le c\left(\frac{\snr{J_{v-M}^{\pm}(\kappa,r)}}{\snr{B_{r}}}\right)^{\iota}\int_{J^{\pm}_{v-M}(\kappa,r)}H\left(x,\frac{(v-M-\kappa)_{\pm}}{r-\rr}\right) \ \dx,
\end{flalign}
for $c=c(\textit{\data},\nr{Dv}_{L^{p}(\Omega)})$ and $\iota=\iota(n,p,q)=\frac{d_{2}-1}{d_{2}}$ with $d_2>1$ from Lemma~\ref{sopo}. Further, in the "$+$" case of \eqref{dg2} we  always take $\kappa\ge 0$, while for the "$-$" occurrence $\kappa\in \mathbb{R}$. Let $0<\rr<r\le 1$ and set $\theta:=\frac{\rr+r}{2}$. Notice that if $\snr{J_{v-M}^{\pm}(\kappa,r)}\ge \frac{1}{2}\snr{B_{r}}$ there is nothing to prove, therefore we assume that this is not the case. Pick $\eta\in C^{1}_{c}(B_{\theta})$ such that $\chi_{B_{\rr}}\le \eta\le \chi_{B_{\theta}}$ and $\snr{D\eta}\le (r-\rr)^{-1}$ and denote
\begin{flalign*}\tilde{v}_{+}:=\eta(v-M-\kappa)_{+}, \ \kappa\ge 0, \qquad \mbox{and} \qquad \tilde{v}_{-}:=\eta(v-M-\kappa)_{-},  \ \kappa\in \mathbb{R}.\end{flalign*} 
We see that $\snr{D\tilde{v}_{\pm}}\le \snr{D(v-M-\kappa)_{\pm}}+(v-M-\kappa)_{\pm}\snr{D\eta}$. By H\"older's inequality, \eqref{sopo2}, the fact that $\frac{r}{2}\le \theta\le r$, \eqref{28} and \eqref{29}, we obtain
\begin{flalign*}
\int_{J^{\pm}_{v-M}(\kappa,\rr)}&H\left(x,\frac{(v-M-\kappa)_{\pm}}{r}\right) \ \dx\le \int_{J^{\pm}_{v-M}(\kappa,\theta)}H\left(x,\frac{\tilde{v}_{\pm}}{r}\right) \ \dx\nonumber \\
\le &c\snr{J^{\pm}_{v-M}(\kappa,r)}^{\frac{d_{2}-1}{d_{2}}}\snr{B_{r}}^{\frac{1}{d_{2}}}\left(\mint_{B_{\theta}}H\left(x,\frac{\tilde{v}_{\pm}}{\theta}\right)^{d_{2}} \ \dx\right)^{\frac{1}{d_{2}}}\nonumber \\
\le &c\left(\frac{\snr{J^{\pm}_{v-M}(\kappa,\theta)}}{\snr{B_{r}}}\right)^{\iota}\left\{\int_{J^{\pm}_{v-M}(\kappa,\theta)}H\left(x,D(v-M-\kappa)_{\pm}\right)+H\left(x,\frac{(v-M-\kappa)_{\pm}}{r-\rr}\right) \ \dx\right\}\nonumber \\
\le &c\left(\frac{\snr{J^{\pm}_{v-M}(\kappa,\theta)}}{\snr{B_{r}}}\right)^{\iota}\int_{J^{\pm}_{v-M}(\kappa,r)}H\left(x,\frac{(v-M-\kappa)_{\pm}}{r-\rr}\right) \ \dx,
\end{flalign*}
which is \eqref{dg2} with the announced dependencies of the constants. At this point, we split the rest of the proof into three parts: in the first one we provide an upper bound to $(v-M)_{+}$, in the second one we control from above $(v-M)_{-}$, while in the third one we obtain the boundedness of $v$. Our main references are \cite{ditr,hahale}.\\\\
\emph{Step 1: control on the essential supremum.} Let $B_{R}\Subset \Omega$ be any ball with radius $R\le 1$, $M:=\nr{\psi}_{L^{\infty}(\Omega)}+1$ and, for $\tau\in (0,1/2]$, $\kappa>0$ and all $j\in \mathbb{N}$, define
\begin{flalign}\label{rrj}
&\kappa_{j}:=R\kappa\left(1-2^{-j}\right), \quad \rr_{j}:=R\left(\tau+(1-\tau){2^{-j}}\right),\nonumber \\
&B_{j}:=B_{\rr_{j}},\quad Y^{+}_{j}:=\mint_{B_{j}}H\left(x,\frac{(v-M-\kappa)_{+}}{R}\right) \ \dx. 
\end{flalign}
Notice that $\rr_{j}-\rr_{j+1}=\frac{R(1-\tau)}{2^{j+1}}$, $\kappa_{j+1}\ge \kappa_{j}$ and that $\rr_{j}/\rr_{j+1}\le 2$. Using $\eqref{dg2}_{+}$ with $\kappa=\kappa_{j+1}$, $\rr=\rr_{j+1}$ and $r=\rr_{j}$ we get
\begin{flalign*}
Y^{+}_{j+1}\le &\mint_{B_{j+1}}H\left(x,\frac{(v-M-\kappa_{j+1})_{+}}{R(\tau+(1-\tau)2^{-j})}\right) \ \dx\nonumber\\
\le &c\left(\frac{\snr{J^{+}_{v-M}(\kappa_{j+1},\rr_{j})}}{\snr{B_{j}}}\right)^{\iota}\mint_{B_{j}}H\left(x,\frac{(v-M-\kappa_{j+1})_{+}}{R(1-\tau)2^{-(j+1)}}\right) \ \dx\nonumber \\
\le &c\left(\frac{\snr{J^{+}_{v-M}(\kappa_{j+1},\rr_{j})}}{\snr{B_{j}}}\right)^{\iota}\mint_{B_{j}}H\left(x,\frac{(v-M-\kappa_{j})_{+}}{R(1-\tau)2^{-(j+1)}}\right) \ \dx\nonumber \\
\le &c\left(\frac{\snr{J^{+}_{v-M}(\kappa_{j+1},\rr_{j})}}{\snr{B_{j}}}\right)^{\iota}(1-\tau)^{-q}2^{jq}Y_{j}^{+},
\end{flalign*}
for $c=c(\textit{\data},\nr{Dv}_{L^{p}(\Omega)})$. Since on $J^{+}_{v-M}(\kappa_{j+1},\rr_{j})$ there holds $v-M-\kappa_{j}\ge \kappa_{j+1}-\kappa_{j}=2^{-(j+1)}R\kappa$, we can estimate
\begin{flalign*}
\frac{\snr{J^{+}_{v-M}(\kappa_{j+1},\rr_{j})}}{\snr{B_{j}}}\le &\snr{B_{j}}^{-1}\left[H^{-}_{B_{R}}\left(\frac{\kappa}{2^{j+1}}\right)\right]^{-1}\int_{ J^{+}_{v-M}(\kappa_{j+1},\rr_{j})}H\left(x,\frac{\kappa}{2^{j+1}}\right) \ \dx \nonumber\\
\le &\left(\frac{2^{j+1}}{\kappa}\right)^{p}\mint_{B_{j}}H\left(x,\frac{(v-M-\kappa_{j})_{+}}{R}\right) \ \dx=\left(\frac{2^{j+1}}{\kappa}\right)^{p}Y^{+}_{j},
\end{flalign*}
therefore, merging the content of the previous two displays we obtain
\begin{flalign*}
Y^{+}_{j+1}\le &c2^{j(p\iota+q)}\kappa^{-p\iota}(1-\tau)^{-q}(Y_{j}^{+})^{1+\iota},
\end{flalign*}
with $c=c(\data,\nr{Dv}_{L^{p}(\Omega)})$. By \cite[Lemma 7.1]{gi}, we have that $\lim_{j\to \infty}Y^{+}_{j}=0$ provided that $Y^{+}_{0}\le c^{-\frac{1}{\iota}}\kappa^{p}(1-\tau)^{\frac{q}{\iota}}2^{-\frac{p\iota+q}{\iota^{2}}}$. If we choose 
\begin{flalign*}
\kappa:=\left\{1+c^{\frac{1}{\iota}}(1-\tau)^{-\frac{q}{\iota}}2^{\frac{p\iota+q}{\iota^{2}}}\left(\mint_{B_{R}}H\left(x,\frac{(v-M)_{+}}{R}\right) \ \dx\right)\right\}^{\frac{1}{p}},
\end{flalign*}
and notice that $\lim_{j\to \infty}\kappa_{j}=R\kappa$ and $\lim_{j\to \infty}\rr_{j}=\tau R$, by means of Fatou's Lemma we can conclude that
\begin{flalign*}
\mint_{B_{\tau R}}H\left(x,\frac{(v-M-R\kappa)_{+}}{R}\right) \ \dx\le \liminf_{j\to \infty}Y^{+}_{j}=0,
\end{flalign*}
so, after a standard covering argument, we can conclude that $(v-M)_{+}\in L^{\infty}_{\mathrm{loc}}(\Omega)$ and if $\tilde{\Omega}\Subset \Omega$ is any open set, then
\begin{flalign}\label{dg3}
\nr{(v-M)_{+}}_{L^{\infty}(\tilde{\Omega})}\le c(\data,\nr{Dv}_{L^{p}(\Omega)},\nr{H(\cdot,Dv)}_{L^{1}(\Omega)}).
\end{flalign}
\emph{Step 2: control on the essential infimum.} As in \emph{Step 1}, let $B_{R}\Subset \Omega$ be any ball with radius $R\le 1$, $\tau\in (0,1/2]$, $M:=\nr{\psi}_{L^{\infty}(\Omega)}+1$ and, for $\kappa>0$ and $j\in \mathbb{N}$, take $\rr_{j}$, $B_{j}$, $\kappa_{j}$ as in \eqref{rrj} and set
\begin{flalign*}
Y_{j}^{-}:=\mint_{B_j}H\left(x,\frac{(v-M+\kappa_{j})_{-}}{R}\right) \ \dx.
\end{flalign*}
Using $\eqref{dg2}_{-}$ with $\kappa=-\kappa_{j+1}$, $\rr=\rr_{j+1}$ and $r=\rr_{j}$ we have
\begin{flalign*}
Y^{-}_{j+1}\le &\mint_{B_{j+1}}H\left(x,\frac{(v-M+\kappa_{j+1})_{-}}{R(\tau+(1-\tau)2^{-j})}\right) \ \dx\nonumber \\
\le &c\left(\frac{\snr{J^{-}_{v-M}(-\kappa_{j+1},\rr_{j})}}{\snr{B_{j}}}\right)^{\iota}\mint_{B_{j}}H\left(x,\frac{(v-M+\kappa_{j+1})_{-}}{R(1-\tau)2^{-(j+1)}}\right) \ \dx\nonumber \\
\le&c\left(\frac{\snr{J^{-}_{v-M}(-\kappa_{j+1},\rr_{j})}}{\snr{B_{j}}}\right)^{\iota}\mint_{B_{j}}H\left(x,\frac{(v-M+\kappa_{j})_{-}}{R(1-\tau)2^{-(j+1)}}\right) \ \dx\nonumber \\
\le &c\left(\frac{\snr{J^{-}_{v-M}(-\kappa_{j+1},\rr_{j})}}{\snr{B_{j}}}\right)^{\iota}(1-\tau)^{-q}2^{jq}Y^{-}_{j},
\end{flalign*}
with $c(\data,\nr{Dv}_{L^{p}(\Omega)})$. Notice that on $J^{-}_{v-M}(-\kappa_{j+1},\rr_{j})$ there holds $-\kappa_{j}-(v-M)\ge -\kappa_{j}+\kappa_{j+1}=2^{-(j+1)}R\kappa$, so we can bound
\begin{flalign*}
\frac{\snr{J^{-}_{v-M}(-\kappa_{j+1},\rr_{j})}}{B_{j}}\le &\snr{B_{j}}^{-1}\left[H^{-}_{B_{r}}\left(\frac{\kappa}{2^{j+1}}\right)\right]^{-1}\int_{J^{-}_{v-M}(-\kappa_{j+1},\rr_{j})}H\left(x,\frac{\kappa}{2^{j+1}}\right) \ \dx\nonumber \\
\le &\left(\frac{2^{j+1}}{\kappa}\right)^{p}\mint_{B_{j}}H\left(x,\frac{(v-M+\kappa_{j})_{-}}{R}\right) \ \dx=\left(\frac{2^{j+1}}{\kappa}\right)^{p}Y_{j}^{-}.
\end{flalign*}
Collecting the two above estimates we then get
\begin{flalign*}
Y^{-}_{j+1}\le c2^{j(p\iota+q)}\kappa^{-p\iota}(1-\tau)^{-q}(Y^{-}_{j})^{1+\iota},
\end{flalign*}
for $c=c(\data,\nr{Dv}_{L^{p}(\Omega)})$. From \cite[Lemma 7.1]{gi}, we have that $\lim_{j\to \infty}Y^{-}_{j}=0$ if $Y_{0}^{-}\le c^{-\frac{1}{\iota}}\kappa^{p}(1-\tau)^{\frac{q}{\iota}}2^{-\frac{p\iota+q}{\iota^{2}}}$. Choosing
\begin{flalign*}
\kappa:=\left\{1+c^{\frac{1}{\iota}}(1-\tau)^{-\frac{q}{\iota}}2^{\frac{p\iota+q}{\iota^{2}}}\left(\mint_{B_{R}}H\left(x,\frac{(v-M)_{-}}{R}\right) \ \dx\right)\right\}^{\frac{1}{p}},
\end{flalign*}
by Fatou's lemma we obtain
\begin{flalign*}
\mint_{B_{\tau R}}H\left(x,\frac{(v-M+R\kappa)_{-}}{R}\right) \ \dx\le \liminf_{j\to \infty}Y_{j}^{-}=0,
\end{flalign*}
therefore, after covering, we get that $(v-M)_{-}\in L^{\infty}_{\mathrm{loc}}(\Omega)$ and for any open $\tilde{\Omega}\Subset \Omega$ there holds
\begin{flalign}\label{dg4}
\nr{(v-M)_{-}}_{L^{\infty}(\tilde{\Omega})}\le c(\data,\nr{Dv}_{L^{\infty}(\Omega)},\nr{H(\cdot,Dv)}_{L^{1}(\Omega)}).
\end{flalign}
\emph{Step 3: conclusion.} Recall that, in \emph{Step 1} we fixed $M:=\nr{\psi}_{L^{\infty}(\Omega)}+1$, thus, from \eqref{dg3}-\eqref{dg4} we can conclude that $v\in L^{\infty}_{\mathrm{loc}}(\Omega)$ and for any open $\tilde{\Omega}\Subset \Omega$ we have
\begin{flalign*}
\nr{v}_{L^{\infty}(\tilde{\Omega})}\le c(\data,\nr{Dv}_{L^{p}(\Omega)},\nr{H(\cdot,Dv)}_{L^{1}(\Omega)},\nr{\psi}_{L^{\infty}(\Omega)}),
\end{flalign*}
which is \eqref{v-bounded}.
\end{proof}
\begin{remark}\label{r2}\rm 
{In the second part of the proof of {Proposition~\ref{prop:GDG}} 
}we see that the constant bounding the local $L^{\infty}$-norm of $v$  is non-decreasing with respect to $\nr{Dv}_{L^{p}(\Omega)}$, $\nr{H(\cdot,Dv)}_{L^{1}(\Omega)}$. When $v\in \mathcal{K}_{\psi}(\Omega)$, $\psi\in L^{\infty}(\Omega)$, is a solution to \eqref{obs}, then clearly $\psi \in \mathcal{K}_{\psi}(\Omega)$, so it is an admissible competitor to $v$ in \eqref{obs}. Testing \eqref{obs} with $w:=\psi$, using $\eqref{A}_{1,2}$, Young's inequality and reabsorbing terms we obtain
\begin{flalign*}
\int_{\Omega}H(x,Dv) \ \dx\le c\int_{\Omega}H(x,D\psi) \ \dx,
\end{flalign*}
with $c=c(\nu,L,p,q)$. This and the coercivity of $\snr{z}\mapsto H(\cdot,z)$, allow incorporating any dependency from $\nr{Dv}_{L^{p}(\Omega)}$, $\nr{H(\cdot,Dv)}_{L^{1}(\Omega)}$ into the one from $(\nu,L,p,q,\nr{H(\cdot,D\psi)}_{L^{1}(\Omega)})$.
\end{remark}
Proceeding further, we will use the following Harnack inequalities, valid for solutions to the obstacle problem \eqref{obs}, which are supersolutions to \eqref{A0}. 
\begin{proposition}\label{p6}{
Let $\tilde{\Omega}\Subset \Omega$ be an open set and $B_{\rr}\Subset \tilde{\Omega}$ be any ball with $\rr\le 1$. Suppose that assumptions \eqref{pq} and \eqref{A} are satisfied, $\psi\in W^{1,H(\cdot)}(\Omega)\cap L^\infty(\Omega)$ and $g\in W^{1,H(\cdot)}(\Omega)$ are such that $\mathcal{K}_{\psi,g}(\Omega)\not =\emptyset$, and $v\in \mathcal{K}_{\psi,g}(\Omega)$ is a solution to the obstacle problem \eqref{obs}. Assume further that $M>0$ is any constant for which \eqref{28} can be realized.} Then there holds 
\begin{flalign}\label{30}
\eup_{B_{\rr/2}}(v-M)_{+}\le c\left(\mint_{B_{\rr}}(v-M)_{+}^{h} \ dx\right)^{\frac{1}{h}} \ \ \mbox{for all} \ \ h\in (0,\infty)
\end{flalign}
for $c=c(\data,h,\nr{H(\cdot,Dv)}_{L^{1}(\Omega)},\nr{\psi}_{L^{\infty}(\Omega)})$. Moreover, if $\tilde{v}\in W^{1,H(\cdot)}(\Omega)$ is a non-negative supersolution to \eqref{A0}, then
\begin{flalign}\label{32}
\left(\mint_{B_{\rr}}{\tilde{v}}^{h_{-}} \ \dx\right)^{\frac{1}{h_{-}}}\le c\eif_{B_{\rr/2}}{\tilde{v}} \quad \mbox{for some} \ \ h_{-}>0,
\end{flalign}
with $c=c(\data, \nr{H(\cdot,D\tilde{v})}_{L^{1}(\Omega)})$ and $h_{-}=h_{-}(\data, \nr{H(\cdot,D\tilde{v})}_{L^{1}(\Omega)})$.
\end{proposition}
\begin{proof} We start by proving \eqref{30}. We already showed in Proposition~\ref{prop:GDG} that $v$ is bounded, so, exploiting also \eqref{28}, from the content of \cite[Section 5]{hahale} we directly obtain \eqref{30}. Next, being $\tilde{v}$ a supersolution to \eqref{A0}, it solves
\begin{flalign}\label{35}
\int_{\Omega}A(x,D\tilde{v})\cdot Dw \ \dx \ge 0 \ \ \mbox{for all non-negative} \ \ w \in W^{1,H(\cdot)}(\Omega).
\end{flalign}
Fix radii $0<\rr<r\le 1$. Now, if $\eta\in C^{1}_{c}(B_{r})$ is the usual cut-off function with $\chi_{B_{\rr}}\le \eta\le \chi_{B_{r}}$ and $\snr{D\eta}\le (r-\rr)^{-1}$, testing \eqref{35} against $w:=\eta^{q}(\tilde{v}-\kappa)_{-}$, for $\kappa\in\mathbb{R}$, we have
\begin{flalign}\label{36}
\int_{J^{-}_{\tilde{v}}(\rr,\kappa)}H(x,D(\tilde{v}-\kappa)_{-}) \ \dx\le c\int_{J^{-}_{\tilde{v}}(r,\kappa)}H\left(x,\frac{(\tilde{v}-\kappa)_{-}}{r-\rr}\right) \ \dx,
\end{flalign}
with $c=c(\nu,L,p,q)$. Again, \cite[Section 6]{hahale} applies yielding \eqref{32}. Let us note that, due to the structure of $H(\cdot)$, the additional term on the right-hand side appearing in~\cite{hahale} vanishes.
\end{proof}
Finally, let us recall that solutions to \eqref{A0} are Lipschitz-continuous. The proof of this statement can be inferred from~\cite[Theorem 1]{bacomi}, up to minor changes.
\begin{proposition}\cite{bacomi}\label{p0}
Let $v\in W^{1,H(\cdot)}(\Omega)$ be a solution to \eqref{A0} with \eqref{pq} and \eqref{A} in force. Then $v\in W^{1,\infty}_{\mathrm{loc}}(\Omega)$. 
\end{proposition}

\subsection{Proof  of Theorem \ref{T4}} Now we are ready to collect the content of this section and conclude the main properties of solutions to the obstacle problem.

\begin{proof}[Proof  of Theorem \ref{T4}] Existence and uniqueness are given by Proposition \ref{prop:ex-uni}, Gehring's and De Giorgi's assertions from Proposition~\ref{prop:GDG}. It suffices to concentrate on the remaining claims on continuity, $\mathcal{A}_{H(\cdot)}$-harmonicity and H\"older continuity.

\medskip

\noindent\emph{Continuity and $\mathcal{A}_{H(\cdot)}$-harmonicity}. For the transparency of the presentation, we split the proof into three steps.

\medskip

\noindent\emph{Step 1: lower semicontinuity}. By the virtue of \eqref{v-bounded}, $v$ is bounded. For $B_{\rr}(x_{0})\equiv B_{\rr}\Subset \Omega$, set {\[m_{\rr}:=\eif_{x\in B_{\rr}}v(x)\quad\text{ and }\quad M_{\rr}:=\eup_{x\in B_{\rr}}v(x).\]} Then, since $v$ is a solution to \eqref{obs}, then $\tilde{v}:=v-m_{\rr}$ is a non-negative supersolution to \eqref{A0} on $B_{\rr}$. Therefore inequality \eqref{32} applies rendering
\begin{flalign*}
\left(\mint_{B_{\rr}}(v-m_{\rr})^{h_{-}} \ \dx\right)^{\frac{1}{h_{-}}}\le c\,\eif_{x\in B_{\rr/2}}(v(x)-m_{\rr}),
\end{flalign*}
with {\it $c=c(\data,\nr{H(\cdot,Dv)}_{L^{1}(\Omega)},\nr{\psi}_{L^{\infty}(\Omega)})$}. Notice that there is no loss of generality in assuming $h_{-}\in (0,1)$. Then,
\begin{flalign}\label{37}
m_{\rr/2}-m_{\rr}\ge&c\left(\mint_{B_{\rr}}(v-m_{\rr})^{h_{-}} \ \dx\right)^{\frac{1}{h_{-}}}\ge c(M_{\rr}-m_{\rr})^{\frac{h_{-}-1}{h_{-}}}\left(\mint_{B_{\rr}}(v-m_{\rr}) \ \dx\right)^{\frac{1}{h_{-}}},
\end{flalign}
for {\it $c=c(\data,\nr{H(\cdot,Dv)}_{L^{1}(\Omega)},\nr{\psi}_{L^{\infty}(\Omega)})$}. From \eqref{37} we can conclude that
\begin{flalign*}
0\le \mint_{B_{\rr}}(v-m_{\rr}) \ \dx \le c(M_{\rr}-m_{\rr})^{(1-h_{-})}(m_{\rr/2}-m_{\rr})^{h_{-}}\xrightarrow[\rr\to 0]{}0.
\end{flalign*}
Thus, by Lebesgue's differentiation theorem we have
\begin{flalign}\label{38}
v(x_{0})=\lim_{\rr\to 0}\mint_{B_{\rr}(x_{0})}v(x)\ \dx =\lim_{\rr\to 0}\eif_{x\in B_{\rr}(x_{0})}v(x) \ \ \mbox{for a.e.} \ x_{0}\in \Omega.
\end{flalign}
Set $\bar{v}(x_{0}):=\lim_{\rr\to 0}{\eif_{x\in B_{\rr}(x_{0})}v(x)}$ and notice that $\bar{v}(x_{0})=\lim_{\rr\to 0}\eif_{x\in B_{\rr}(x_{0})}\bar{v}(x),$ hence $\bar{v}$ is lower semicontinuous. Identity \eqref{38} then gives that $v$ admits a lower semicontinuous representative.\\\\
\emph{Step 2: continuity}. From now on we will identify the lower semicontinuous representative $\bar{v}$ with $v$. Since $v\in \mathcal{K}_{\psi,g}(\Omega)$ and $\psi\in C(\Omega)$, for all $x_{0}\in \Omega$ there holds
\begin{flalign*}
v(x_{0})=\lim_{\rr\to 0}\eif_{y\in B_{\rr}(x_{0})}v(y)\ge \lim_{\rr\to 0}\eif_{y\in B_{\rr}(x_{0})}\psi(y)=\psi(x_{0}).
\end{flalign*}
Fix $\varepsilon>0$ and $B_{\rr}(x_{0})$ such that $B_{4\rr}(x_{0})\Subset \Omega$. Thus, from the continuity of $\psi$ and the lower semicontinuity of $v$, for $\rr$ sufficiently small we have
\begin{flalign}\label{40}
\sup_{x\in B_{\rr}(x_{0})}\psi(x)\le \varepsilon+ v(x_{0}) \ \ \mbox{and} \ \ \inf_{x\in B_{\rr}(x_{0})}v(x)>v(x_{0})-\varepsilon.
\end{flalign}
We apply \eqref{30} to $v-M$ with $M:=v(x_{0})+\varepsilon$, which is admissible by $\eqref{40}_{1}$, to get
\begin{flalign*}
\eup_{x\in B_{\rr/2}(x_{0})}&\left(v(x)-(v(x_{0})+\varepsilon)\right)\le \eup_{x\in B_{\rr/2}(x_{0})}\left(v(x)-(v(x_{0})+\varepsilon)\right)_{+}\nonumber \\
\le &c\mint_{B_{\rr}(x_{0})}\left(v(x)-(v(x_{0})+\varepsilon)\right)_{+} \ \dx\le c\mint_{B_{\rr}(x_{0})}(v(x)-v(x_{0})+\varepsilon) \ \dx\nonumber \\
=&c\left(\mint_{B_{\rr}(x_{0})} v(x) \ \dx- v(x_{0})\right)+c\varepsilon,
\end{flalign*}
for {\it $c=c(\data,\nr{H(\cdot,Dv)}_{L^{1}(\Omega)},\nr{\psi}_{L^{\infty}(\Omega)})$}. We also used $\eqref{40}_{2}$ in the last inequality above. We can rearrange the content of the previous display in a more convenient way:
\begin{flalign}\label{39}
\eup_{x\in B_{\rr/2}(x_{0})}v(x)\le c\left(\mint_{B_{\rr}(x_{0})}v(x) \ \dx-v(x_{0})\right)+(1+c)\varepsilon+v(x_{0}).
\end{flalign}
Recall that we are considering the lower semicontinuous representant of $v$, so \emph{for every} $x_{0}\in \Omega$, 
$$\lim_{\rr\to 0}\mint_{B_{\rr}(x_{0})}v(x) \ \dx=v(x_{0}),$$ 
thus we can send $\rr\to 0$ and $\varepsilon\to 0$ in \eqref{39} to conclude
\begin{flalign*}
\lim_{\rr\to 0}\eup_{x\in B_{\rr}(x_{0})}v(x)\le v(x_{0}).
\end{flalign*}
This and the lower semicontinuity proved in \emph{Step 1} render that $v$ is continuous.\\\\
\emph{Step 3: {$\mathcal{ A}_{H(\cdot)}$-harmonicity outside of} the contact set.} Define 
\begin{flalign}\label{om}
\Omega_{0}:=\left\{x \in \Omega\colon v(x)>\psi(x)\right\}, \ \ \Omega_{c}:=\left\{x\in \Omega\colon v(x)=\psi(x)\right\},
\end{flalign}
and pick $\eta \in C^{\infty}_{c}(\Omega_{0})$. Then, using the continuity of $v$ and $\psi$, for any $x_{0}\in \supp(\eta)$ we obtain
\begin{flalign*}
v(x_{0})-\psi(x_{0})\ge \min_{x\in \supp (\eta)}(v(x)-\psi(x))\ge \frac{\min_{x\in \supp (\eta)}(v(x)-\psi(x))}{2\max_{x\in\supp (\eta)}\snr{\eta(x)}}\max_{x\in \supp (\eta)}\snr{\eta(x)}.
\end{flalign*}
Fix $s_{0}:= \frac{\min_{x\in \supp (\eta)}(v(x)-\psi(x))}{2\max_{x\in\supp\, (\eta)}\snr{\eta(x)}}$, and notice that, for all $s\in (-s_{0},s_{0})$, $v+s\eta\ge \psi$. Hence,  $w:=v+s\psi$ is admissible in \eqref{obs}, so we get
\begin{flalign*}
\int_{\Omega_{0}}A(x,Dv)\cdot D(s\eta) \ \dx \ge 0.
\end{flalign*}
Since $s$ can be either positive or negative {and $\eta$ is arbitrary}, we can conclude that $v$ solves
\begin{flalign*}
\int_{\Omega_{0}}A(x,Dv)\cdot D\eta \ \dx =0 \ \ \mbox{for all} \ \ \eta \in C^{\infty}_{c}(\Omega_{0}).
\end{flalign*}
\noindent\emph{H\"older regularity}. By assumption, $\psi\in C^{0,\beta_{0}}(\Omega)$ for some $\beta_{0}\in (0,1]$. Let $ B_{\rr}\equiv B_{\rr}(x_{0})$ be any ball with $0<\rr\le 1$ and so that $B_{4\rr}\Subset \Omega$. Notice that, by \emph{Step 2}, the set $\Omega_{0}$ is open since $v$ is continuous. If $B_{\rr}\cap \Omega_{c}=\emptyset$, by \emph{Step 3}, $v$ solves \eqref{A0} in $\Omega_{0}$. Proposition \ref{p0} then applies, so we get that $v\in C^{0,\beta}(B_{\rr})$ for any $\beta\in (0,1]$. On the other hand, if $B_{\rr}\cap \Omega_{c}\not=\emptyset$, we pick $x_{+}$, $x_{-}\in \bar{B}_{4\rr}$ so that $\psi(x_{+})=\sup_{x\in B_{4\rr}}\psi(x)$ and $\psi(x_{-})=\inf_{x\in B_{4\rr}}\psi(x)$. Now set
\begin{flalign*}
\vartheta_{+}:=\osc_{x\in B_{2\rr}}\psi(x)+\inf_{x\in B_{2\rr}}v(x) \quad \mbox{and} \quad \vartheta_{-}:=\osc_{x\in B_{2\rr}}\psi(x)-\inf_{x\in B_{2\rr}}v(x),
\end{flalign*}
and let $\bar{x}\in \left(\Omega_{c}\cap B_{\rr}\right)$.
Since
\begin{flalign*}
\inf_{x\in B_{2\rr}}\psi(x)\le \inf_{x\in B_{2\rr}}v(x)\le v(\bar{x})=\psi(\bar{x})\le \sup_{x\in B_{2\rr}}\psi(x),
\end{flalign*}
we have that {$\nr{\psi}_{L^{\infty}(B_{2\rr})}\le \vartheta_{+}$}. Therefore, by \eqref{30} with $M=\vartheta_{+}$ we have
\begin{flalign}\label{c2}
\sup_{x\in B_{\rr}}(v(x)-\vartheta_{+})_{+}\le c\left(\mint_{B_{2\rr}}(v(x)-\vartheta_{+})_{+}^{h} \ \dx\right)^{\frac{1}{h}} \ \ \mbox{for all} \ \ h \in (0,\infty),
\end{flalign}
with {\it $c=c(\data,\nr{H(\cdot,Dv)}_{L^{1}(\Omega)},\nr{\psi}_{L^\infty(\Omega)},h)$}. We stress that, for all $x\in B_{2\rr}$,
\begin{flalign}\label{c1}
\begin{cases}
\ v(x)-\vartheta_{+}\le v(x)+\osc_{x\in B_{2\rr}}\psi(x)-\inf_{x\in B_{2\rr}}v(x)=v(x)+\vartheta_{-},\\
\ v(x)+\vartheta_{-}\ge 0.
\end{cases}
\end{flalign}
From \eqref{c2} with $h=h_{-}$ (the exponent appearing in \eqref{32}) and $\eqref{c1}$ we immediately get
\begin{flalign}\label{c3}
\sup_{x\in B_{\rr}}(v(x)-\vartheta_{+})_{+}\le c\left(\mint_{B_{2\rr}}(v(x)+\vartheta_{-})^{h_{-}} \ \dx\right)^{\frac{1}{h_{-}}},
\end{flalign}
where {\it $c=c(\data,\nr{H(\cdot,Dv)}_{L^{1}(\Omega)},\nr{\psi}_{L^\infty(\Omega)})$}. Moreover, by $\eqref{c1}_{2}$ we also see that $v+\vartheta_{-}$ is a non-negative supersolution to \eqref{A0} in $B_{2\rr}$, thus, using the definition of $\vartheta_{+}$, \eqref{c3} and \eqref{32}  we have
\begin{flalign*}
\osc_{x\in B_{\rr}}v(x)-\osc_{x\in B_{2\rr}}\psi(x)\le &\sup_{x\in B_{\rr}}(v(x)-\vartheta_{+})_{+}\le c\inf_{x\in B_{\rr}}(v(x)+\vartheta_{-})\nonumber \\
\le&c\left(\inf_{x\in B_{\rr}}v(x)+\osc_{x\in B_{2\rr}}\psi(x)-\inf_{x\in B_{2\rr}}v(x)\right)\nonumber \\
\le &c\left(v(\bar{x})+\osc_{x\in B_{2\rr}}\psi(x)-\inf_{x\in B_{2\rr}}\psi(x)\right)\nonumber \\
\le &c\left(\psi(\bar{x})+\osc_{x\in B_{2\rr}}\psi(x)-\inf_{x\in B_{2\rr}}\psi(x)\right)\le c\osc_{x\in B_{2\rr}}\psi(x),
\end{flalign*}
for {\it $c=c(\data,\nr{H(\cdot,Dv)}_{L^{1}(\Omega)},\nr{\psi}_{L^\infty(\Omega)})$}. The content of the above display renders
\begin{flalign*}
\osc_{x\in B_{\rr}}v(x)\le c\osc_{x\in B_{2\rr}}\psi(x),
\end{flalign*}
i.e. $[v]_{0,\beta_{0};B_{\rr}}\le c$ with {\it $c=c(\data,\nr{H(\cdot,Dv)}_{L^{1}(\Omega)},\nr{\psi}_{L^{\infty}(\Omega)},[\psi]_{0,\beta_{0}})$}. After covering, we can conclude that $v \in C^{0,\beta_{0}}_{\mathrm{loc}}(\Omega)$.
\end{proof}

As a direct consequence of `\emph{Continuity and $\mathcal{A}_{H(\cdot)}$-harmonicity}' of Theorem \ref{T4} we infer the existence of a continuous extension of solutions to \eqref{A0}.
\begin{corollary}\label{coro-cont}
If $u \in W^{1,H(\cdot)}(\Omega)$ is a solution to \eqref{A0}, then there exists $v\in C(\Omega)$ such that $u=v$ almost everywhere.
\end{corollary}
\begin{proof}
Any solution $u$ to \eqref{A0} can be seen as a solution to problem \eqref{obs} in $\mathcal{K}_{-\infty,u}(\Omega)$ with obstacle constantly equal to $-\infty$. Then the third part of Theorem \ref{T4} applies rendering that $u$ can be redefined on sets of zero $n$-dimensional Lebesgue measure so that it becomes continuous.
\end{proof}
\section{Removable sets}\label{sec:rem}
In this last section we prove our main result, i.e. Theorem \ref{T6}. 
\subsection{Auxiliary results}
To begin, we show a Caccioppoli-type inequality for non-negative supersolutions to \eqref{A0}.
\begin{lemma}\label{lem-cacc}
Under assumptions \eqref{pq} and \eqref{A}, let $B_{\rr}\Subset \Omega$ be any ball, $\tilde{v}\in W^{1,H(\cdot)}(\Omega)$ a supersolution to \eqref{A0}, non-negative in $B_{\rr}$ and $\eta\in C^{1}_{c}(B_{\rr})$. Then for all $\gamma\in(1,p)$ there holds
\begin{flalign*}
\int_{B_{\rr}}\tilde{v}^{-\gamma}\eta^{q}H(x,D\tilde{v}) \ \dx\le c\int_{B_{\rr}}\tilde{v}^{-\gamma}H(x,\snr{D\eta}\tilde{v}) \ \dx,
\end{flalign*}
with $c=c(\nu,L,p,q,\gamma)$.
\end{lemma}
\begin{proof}
Since $\tilde{v}$ is a non-negative supersolution to \eqref{A0}, inequality \eqref{32} applies, so, either $\tilde{v}\equiv0$ a.e. on $B_{2\rr}$, or we can assume that $\tilde{v}$ is strictly positive in $B_{\rr}$. In the first scenario there is nothing interesting to prove, so we can look at the second one. For $\eta$ as in the statement, and any $\tilde{\gamma}>0$, we test \eqref{sux} against $w:=\eta^{q}\tilde{v}^{-\tilde{\gamma}}$ to obtain, with the help of $\eqref{A}_{1,2}$ and Young's inequality,
\begin{flalign}\label{41}
\nu \tilde{\gamma}\int_{B_{\rr}}\tilde{v}^{-\tilde{\gamma}-1}\eta^{q}H(x,D\tilde{v}) \ \dx \le &Lq\int_{B_{\rr}}\left(\frac{H(x,D\tilde{v})}{\snr{D\tilde{v}}}\eta^{q-1}\snr{D\eta}\tilde{v}\right)\tilde{v}^{-\tilde{\gamma}-1} \ \dx\nonumber \\
\le &\frac{\nu \tilde{\gamma}}{2}\int_{B_{\rr}}\tilde{v}^{-\tilde{\gamma}-1}H(x,D\tilde{v})\eta^{q} \ \dx \nonumber \\
&+\left(\frac{c}{\nu\tilde{\gamma}}\right)^{q-1}\int_{B_{\rr}}H(x,\snr{D\eta}\tilde{v}) \tilde{v}^{-\tilde{\gamma}-1} \ \dx,
\end{flalign}
for $c=c(L,p,q)$. Absorbing terms in \eqref{40} and setting $\gamma:=\tilde{\gamma}+1$, we obtain the announced inequality.
\end{proof}
Now we show how to control the oscillation of a solution $v\in \mathcal{K}_{\psi}(\Omega)$ across the contact set via the oscillation of the obstacle $\psi$.
\begin{lemma}\label{osc}
Under assumptions \eqref{pq} and \eqref{A}, let $K\subset \Omega$ be a compact set and $v\in \mathcal{K}_{\psi}(\Omega)$ be a solution to problem \eqref{obs} with obstacle $\psi\in C(\Omega)$ and such that
\begin{flalign}\label{43}
\snr{\psi(x_{1})-\psi(x_{2})}\le C_\psi\snr{x_{1}-x_{2}}^{\beta_{0}} \ \ \mbox{for all} \ \ x_{1}\in K,\  x_{2}\in \Omega,
\end{flalign}
where $\beta_{0}\in (0,1]$ and $C_\psi$ is a positive, absolute constant. Let $\mu=-\diver A(x,Dv)$. Then, for any $\rr\in \left(0,\frac{1}{40}\min\left\{1,\dist\{K,\partial\Omega\}\right\}\right)$ and all $\bar{x}\in K$ it holds 
\begin{flalign*}
\mu(B_{\rr}(\bar{x}))\le c(\data_{\psi},C_\psi,\nr{a}_{L^\infty(\Omega)})\int_{B_{\rr}(\bar{x})}{H_{\sigma}(x,\rr^{-1})} \ \dx,
\end{flalign*}
where $\sigma:= 1-\tfrac{\beta_{0}}{q}(p-1)$.
\end{lemma}
\begin{proof}
Since $v\in \mathcal{K}_{\psi}(\Omega)$ is a solution to problem \eqref{obs}, it is a supersolution to \eqref{A0} and, given that $\psi \in C(\Omega)$, by Theorem \ref{T4}, third part, $v$ is continuous. Since $v$ realizes \eqref{sux}, then Riesz's representation theorem renders the existence of a unique, non-negative Radon measure $\mu$ such that for all $\eta \in C^{\infty}_{c}(\Omega)$ there holds
\begin{flalign}\label{meas}
\int_{\Omega}A(x,Dv)\cdot D\eta \ \dx -\int_{\Omega}\eta \ \d\mu=0 \ \ \mbox{in} \ \ \Omega.
\end{flalign}
Let $\bar{x}\in K$ and $B_{\rr}(\bar{x})\subset \Omega$ such that $B_{16\rr}(\bar{x})\Subset \Omega$. Keeping in mind the terminology introduced in \eqref{om}, if {$B_{\rr}(\bar{x})$ does not touch the contact set,} i.e. $B_{\rr}(\bar{x})\cap\Omega_{c}=\emptyset$, then, by Theorem \ref{T4}, the third part, the function $v$ is $\mathcal{A}_{H(\cdot)}$-harmonic in $B_{\rr}(\bar{x})$ and $\mu(B_{\rr}(\bar{x}))\equiv 0$. Hence, it suffices to consider only the case when $B_{\rr}\cap \Omega_{c}\not = \emptyset$. Let $x_{0}\in B_{\rr}(\bar{x})\cap\Omega_{c}$ and notice that, by monotonicity, $\mu(B_{\rr}(\bar{x}))\le \mu(B_{2\rr}(x_{0}))$. Let $x_{+}, x_{-}\in \bar{B}_{16\rr}(x_{0})$ be such that $\psi(x_{+})=\sup_{x\in B_{16\rr}(x_{0})}\psi(x)$ and $\psi(x_{-})=\inf_{x\in B_{16\rr}(x_{0})}\psi(x)$. Then, by \eqref{43},
\begin{flalign}\label{42}
\osc_{x\in B_{16\rr}(x_{0})}\psi(x)=&\psi(x_{+})-\psi(x_{-})\pm \psi(\bar{x})\nonumber\\
\le &C_\psi\left(\snr{x_{+}-\bar{x}}^{\beta_{0}}+\snr{\bar{x}-x_{-}}^{\beta_{0}}\right)\nonumber \\
\le &C_\psi\left(\snr{x_{+}-x_{0}}^{\beta_{0}}+2\snr{\bar{x}-x_{0}}^{\beta_{0}}+\snr{x_{0}-x_{-}}^{\beta_{0}}\right)\le c(\beta_{0},C_\psi)\rr^{\beta_{0}}.
\end{flalign}
Now set
\begin{flalign*}
\vartheta_+:=\osc_{x\in B_{8\rr}(x_{0})}\psi(x)+\inf_{x\in B_{8\rr}(x_{0})}v(x)\quad \mbox{and}\quad \vartheta_{-}:=\osc_{x\in B_{8\rr}(x_{0})}\psi(x)-\inf_{x\in B_{8\rr}(x_{0})}v(x),
\end{flalign*}
and observe that   for all $x\in B_{8\rr}(x_{0})$,
\begin{flalign}\label{44}
\begin{cases}
\ v(x)-\vartheta_{+}\le v(x)+
\osc_{x\in B_{8\rr}(x_{0})}\psi(x)-\inf_{x\in B_{8\rr}(x_{0})}v(x)=v(x)+\vartheta_{-},\\
\ v(x)+\vartheta_{-}\ge 0.
\end{cases}
\end{flalign}
Moreover, since 
\begin{flalign*}
\inf_{x\in B_{8\rr}(x_{0})}\psi(x)\le \inf_{x\in B_{8\rr}(x_{0})}v(x)\le v(x_{0})=\psi(x_{0})\le \sup_{x\in B_{8\rr}(x_{0})}\psi(x),
\end{flalign*}
we infer that $ \nr{\psi}_{L^{\infty}(B_{8\rr}(x_{0}))}\le \vartheta_+$. Therefore, by \eqref{30} with $M=\vartheta_{+}$ we have
\begin{flalign*}
\sup_{x\in B_{4\rr}(x_{0})}(v(x)-\vartheta_{+})_{+}\le c\left(\mint_{B_{8\rr}}(v(x)-\vartheta_{+})_{+}^{h} \ \dx\right)^{\frac{1}{h}},
\end{flalign*}
for {\it $c=c(\data_{\psi},h)$}. Combining the content of the previous display with~\eqref{44} we get
\begin{flalign*}
\sup_{x\in B_{4\rr}(x_{0})}(v(x)-\vartheta_{+})_{+}\le c\left(\mint_{B_{8\rr}}(v(x)+\vartheta_{-})^{h-} \ \dx\right)^{\frac{1}{h-}},
\end{flalign*}
where {\it $c=c(\data_{\psi})$} and $h_{-}$ is the exponent appearing in \eqref{32}. From $\eqref{44}_{2}$, we see that $v+\vartheta_{-}$ is a non-negative supersolution to \eqref{A0} in $B_{8\rr}(x_{0})$, thus, using the definition of $\vartheta_{+}$ we have
\begin{flalign*}
\osc_{x\in B_{4\rr}(x_{0})}v(x)&-\osc_{x\in B_{8\rr}(x_{0})}\psi(x)\le\sup_{x\in B_{4\rr}(x_{0})}(v(x)-\vartheta_{+})_{+}\le c\inf_{x\in B_{4\rr}(x_{0})}(v(x)+\vartheta_{-})\nonumber \\
\le &c\left(v(x_{0})+\osc_{x\in B_{8\rr}(x_{0})}\psi(x)-\inf_{x\in B_{8\rr}(x_{0})}\psi(x)\right)
\le c\osc_{x\in B_{8\rr}(x_{0})}\psi(x),
\end{flalign*}
which due to \eqref{42}  implies that
\begin{flalign}\label{45}
\osc_{x\in B_{4\rr}(x_{0})}v(x)\le c\osc_{x\in B_{8\rr}(x_{0})}\psi(x)\le c\rr^{\beta_{0}}.
\end{flalign}
 Here {\it $c=c(\data_{\psi},C_\psi)$}. Now, set 
 \begin{equation}
     \label{v-tilde}\tilde{v}:=v-\inf_{x\in B_{4\rr}(x_{0})}v(x),
 \end{equation} notice that $D\tilde{v}=Dv$ and pick $\eta\in C^{1}_{c}(B_{4\rr}(x_{0}))$ such that $\chi_{B_{2\rr}(x_{0})}\le \eta\le \chi_{B_{4\rr}(x_{0})}$ and $\snr{D\eta}\le \rr^{-1}$. Clearly, recalling also Theorem \ref{T4}, $\tilde{v}$ is a bounded supersolution to \eqref{A0} which is non-negative in $B_{4\rr}(x_{0})$, thus, by \eqref{meas}, the weighted H\"older's inequality and the fact that $(q-1)p/(p-1)\ge q$, we obtain
\begin{flalign}\label{46}
\mu(B_{2\rr}(x_{0}))\le& \int_{B_{4\rr}(x_{0})}\eta^{q} \ \d\mu=q\int_{B_{4\rr}(x_{0})}\eta^{q-1}A(x,Dv)\cdot D\eta \ \dx\nonumber\\ 
\le&qL\int_{B_{4\rr}(x_{0})}\eta^{q-1}\left(\tilde{v}^{(-1+1)\gamma\frac{p-1}{p}}\snr{Dv}^{p-1}+\tilde{v}^{(-1+1)\gamma\frac{q-1}{q}}a(x)\snr{Dv}^{q-1}\right)\snr{D\eta} \ \dx\nonumber \\
\le &qL\left(\int_{B_{4\rr}(x_{0})}\eta^{q}\tilde{v}^{-\gamma}\snr{Dv}^{p} \ \dx\right)^{\frac{p-1}{p}}\left(\int_{B_{4\rr}(x_{0})}\tilde{v}^{\gamma(p-1)}\snr{D\eta}^{p} \ \dx \right)^{\frac{1}{p}}\nonumber \\
&+qL\left(\int_{B_{4\rr}(x_{0})}\eta^{q}\tilde{v}^{-\gamma}a(x)\snr{Dv}^{q} \ \dx\right)^{\frac{q-1}{q}}\left(\int_{B_{4\rr}(x_{0})}\tilde{v}^{\gamma(q-1)}a(x)\snr{D\eta}^{q} \ \dx\right)^{\frac{1}{q}},
\end{flalign}
where $\gamma\in (1,p)$. The reason behind this choice will be clear in a few lines. Let us estimate the last two terms in \eqref{46}. 
First, notice that
since $B_{\rr}(\bar{x})\subset B_{2\rr}(x_{0})$, then $a_{i}(B_{\rr}(\bar{x}))\ge a_{i}(B_{4\rr}(x_{0}))$. Invoking also \eqref{pq}, we have
\begin{flalign}\label{49}
\int_{B_{4\rr}(x_{0})}H(x,\rr^{-1}) \ \dx =&\int_{B_{4\rr}(x_{0})}H(x,\rr^{-1})\pm a_{i}(B_{4\rr}(x_{0})) \ic{\rr^{-q}}\ \dx\nonumber \\
\le &c\rr^{n}H^{-}_{B_{4\rr}(x_{0})}(\rr^{-1})\le c\rr^{n}H^{-}_{B_{\rr}(\bar{x})}(\rr^{-1})\nonumber \\
\le &c\int_{B_{\rr}(\bar{x})}H(x,\rr^{-1}) \ \dx,
\end{flalign}
where $c=c(n,[a]_{0,\alpha})$. We continue estimating \eqref{46} by means of Lemma \ref{lem-cacc}, \eqref{v-tilde}, \eqref{45} and \eqref{49} getting
\begin{flalign}\label{47}
&\left(\int_{B_{4\rr}(x_{0})}\eta^{q}\tilde{v}^{-\gamma}\snr{Dv}^{p} \ \dx\right)^{\frac{p-1}{p}}\left(\int_{B_{4\rr}(x_{0})}\tilde{v}^{\gamma(p-1)}\snr{D\eta}^{p} \ \dx \right)^{\frac{1}{p}}\nonumber \\
&\le c\left(\int_{B_{4\rr}(x_{0})}\tilde{v}^{-\gamma}H(x,\snr{D\eta}\tilde{v}) \ \dx\right)^{\frac{p-1}{p}}\left(\int_{B_{4\rr}(x_{0})}\tilde{v}^{\gamma(p-1)}\snr{D\eta}^{p} \ \dx \right)^{\frac{1}{p}}\nonumber \\
&\le c\left(\int_{B_{4\rr}(x_{0})}\tilde{v}^{p-\gamma}H(x,\rr^{-1}) \ \dx\right)^{\frac{p-1}{p}}\left(\int_{B_{4\rr}(x_{0})}\tilde{v}^{\gamma(p-1)}\snr{D\eta}^{p} \ \dx \right)^{\frac{1}{p}}\nonumber \\
&\le c\left(\osc_{x\in B_{4\rr}(x_{0})}v(x)\right)^{p-1}\int_{B_{4\rr}(x_{0})}H(x,\rr^{-1}) \ \dx\nonumber \\
&\le c\left(\osc_{x\in B_{8\rr}(x_{0})}\psi(x)\right)^{p-1}\int_{B_{\rr}(\bar{x})}H(x,\rr^{-1}) \ \dx\le c\rr^{\beta_{0}(p-1)}\int_{B_{\rr}(\bar{x})}H(x,\rr^{-1}) \ \dx
\end{flalign}
and similarly
\begin{flalign}\label{48}
&\left(\int_{B_{4\rr}(x_{0})}\eta^{q}\tilde{v}^{-\gamma}a(x)\snr{Dv}^{q} \ \dx\right)^{\frac{q-1}{q}}\left(\int_{B_{4\rr}(x_{0})}\tilde{v}^{\gamma(q-1)}a(x)\snr{D\eta}^{q} \ \dx\right)^{\frac{1}{q}}\nonumber \\
&\le c\left(\int_{B_{4\rr}(x_{0})}\tilde{v}^{-\gamma}H(x,\snr{D\eta}\tilde{v}) \ \dx\right)^{\frac{q-1}{q}}\left(\int_{B_{4\rr}(x_{0})}\tilde{v}^{\gamma(q-1)}a(x)\snr{D\eta}^{q}\ \dx\right)^{\frac{1}{q}}\nonumber \\
&\le c\left(\osc_{x\in B_{4\rr}(x_{0})}v(x)\right)^{\frac{p(q-1)}{q}}\int_{B_{\rr}(\bar{x})}H(x,\rr^{-1})\ \dx\nonumber \\
&\le c\left(\osc_{x\in B_{8\rr}(x_{0})}\psi(x)\right)^{\frac{p(q-1)}{q}}\int_{B_{\rr}(\bar{x})}H(x,\rr^{-1})\ \dx\nonumber \\
&\le c\rr^{\beta_{0}\frac{p(q-1)}{q}}\int_{B_{\rr}(\bar{x})}H(x,\rr^{-1}) \ \dx \le c\rr^{\beta_{0}(p-1)}\int_{B_{\rr}(\bar{x})}H(x,\rr^{-1}) \ \dx, \end{flalign}
with {\it $c=c(\data_{\psi},C_{\psi})$}. Here, we also used{ that  $\beta_{0}\frac{p(q-1)}{q}>\beta_{0}(p-1)$.  In fact notice that, since $q> p$, we have also
\begin{flalign}\label{c0}
\begin{cases}
\ 1>1-\frac{\beta_{0}}{q}(p-1)> 1-\frac{\beta_{0}}{p}(p-1),\\
\ p\left(1-\frac{\beta_{0}}{q}(p-1)\right)>1.
\end{cases}
\end{flalign}} Merging the content of displays \eqref{46}, \eqref{47} and \eqref{48}, we obtain
\begin{flalign}\label{65}
\mu(B_{\rr}(\bar{x}))\le&\mu(B_{2\rr}(x_{0}))\le c\rr^{\beta_{0}(p-1)}\int_{B_{\rr}(\bar{x})}H(x,\rr^{-1})\ \dx,
\end{flalign}
for {\it $c=c(\data_{\psi},C_{\psi})$}. We stress that when the modulating coefficient $a$ degenerates in the sense of \eqref{adeg}, we find again the result in \cite[Lemma 2.1]{kizo}. Now we need to relate the quantity in \eqref{65} with the intrinsic Hausdorff measures discussed in section \ref{haus}. {Having~\eqref{c0} and} recalling that $\rr\le 1$, we conclude that
\begin{flalign*}
\rr^{\beta_{0}(p-1)}\int_{B_{\rr}(\bar{x})}H(x,\rr^{-1}) \ \dx \le c\int_{B_{\rr}(\bar{x})}\rr^{-p\left(1-\frac{\beta_{0}}{q}(p-1)\right)}+a(x)^{1-\frac{\beta_{0}}{q}(p-1)}\rr^{-q\left(1-\frac{\beta_{0}}{q}(p-1)\right)} \ \dx,
\end{flalign*}
with {\it $c=c(p,q,\nr{a}_{L^\infty(\Omega)})$}.
We set $\sigma:=1-\frac{\beta_{0}}{q}(p-1)$ and $H_{\sigma}(x,t)=t^{p\sigma}+a(x)^{\sigma}t^{q\sigma}$, {see~\eqref{Hgamma}}. Keeping in mind $\eqref{c0}_{2}$, we see that the Hausdorff-type measures defined by means of $H_{\sigma}(\cdot)$ enters among those discussed in Section \ref{haus}, therefore, from \eqref{65} we obtain
\begin{flalign*}
\mu(B_{\rr}(\bar{x}))\le c\int_{B_{\rr}(\bar{x})}H_{\sigma}(x,\rr^{-1}) \ \dx,
\end{flalign*}
for {\it $c=c(\data_{\psi},C_\psi,\nr{a}_{L^\infty(\Omega)})$.} Notice that the dependency of the constants from $\data_{\psi}$, is justified by Remark \ref{r2}.
\end{proof}

\subsection{Proof of Theorem \ref{T6}}
Now we are in position for proving our main result on the removability of singularities for solutions to \eqref{A0}. Fix an open set $U\Subset \Omega$ with $U\cap E \not =\emptyset$ (otherwise the proof is trivial), and let $v\in \mathcal{K}_{u}(U)$ be the unique solution to problem \eqref{obs} in $U$. Notice that, in the light of the discussion at the beginning of Section \ref{o}, $\mu=-\diver A(x,Dv)$ is a non-negative Radon measure by Riesz's representation theorem. Let $K\Subset E\cap U$ be a compact set. Lemma \ref{osc} then assures that for any $x_{0}\in K$ and all $\rr\in \left(0,\frac{1}{40}\min\left\{1,\dist\{K,\partial (E\cap U)\}\right\}\right)$, the following decay holds 
\begin{flalign}\label{51}
\mu(B_{\rr}(x_{0}))\le c\int_{B_{\rr}(x_{0})}H_{\sigma}\left(x,{\rr}^{-1}\right) \ \dx,
\end{flalign}
with $c=c(\texttt{data}_{u},C_{u},\nr{a}_{L^{\infty}(\Omega)})$. By assumption, $\mathcal{H}_{H_{\sigma}(\cdot)}(E)=0$, {consequently also}  $\mathcal{H}_{H_{\sigma}(\cdot)}(K)=0$. Therefore, for any $\varepsilon>0$ we can cover $K$ with balls $B_{\rr_{j}}(x_{j})$ with radii $\rr_{j}$ less than $\frac{1}{80}\min\left\{1,\dist\{K,\partial (E\cap U)\}\right\}$ and $(x_{j})_{j \in \mathbb{N}}\subset K$, so that
\begin{flalign*}
\mu(K)\le \sum_{j=1}^{\infty}\mu(B_{\rr_{j}}(x_{j}))\stackrel{\eqref{51}}{\le} c\sum_{j=1}^{\infty}\int_{B_{\rr_{j}}(x_{j})}H_\sigma\left(x,{\rr}^{-1}\right) \ \dx\le \varepsilon.
\end{flalign*}
Sending $\varepsilon$ to zero in the previous display, we can conclude that $\mu(K)=0$. Moreover, since $\mu$ is a  Radon measure, $K$ is an arbitrary compact subset of $E\cap U$, and $E\cap U$ is $\mu$-measurable, we get that $\mu(E\cap U)=0$. Now we turn our attention to the set $U\setminus E$ and show that $\mu(U\setminus E)=0$. Let $\eta\in W^{1,H(\cdot)}_{0}(U\setminus E)$, $\eta(x)\ge 0$ for a.e. $x\in U\setminus E$ and, for $s>0$, set $\eta_{s}:=\min\left\{s\eta,v-u\right\}$. Notice that, $\eta_{s}\in W^{1,H(\cdot)}_{0}((U\setminus E)\cap \{x\colon v(x)>u(x)\})$, therefore, by Theorem \ref{T4}, third part,
\begin{flalign}\label{x1} 
\int_{U\setminus E}A(x,Dv)\cdot D\eta_{s} \ \dx=\int_{(U\setminus E)\cap \{v>u\}}A(x,Dv)\cdot D\eta_{s} \ \dx=0.
\end{flalign}
On the other hand, since $u$ solves \eqref{A0} in $U\setminus E$, then
\begin{flalign}\label{x2}
\int_{U\setminus E}A(x,Du)\cdot D\eta_{s} \ \dx =0.
\end{flalign}
Subtracting \eqref{x2} from \eqref{x1} we end up with
\begin{flalign}\label{x3}
\int_{U\setminus E}\left(A(x,Dv)-A(x,Du)\right)\cdot D\eta_{s} \ \dx=0.
\end{flalign}
Then, recalling also \eqref{A-monotone}, from \eqref{x1}-\eqref{x3} we obtain
\begin{flalign*}
\int_{(U\setminus E)\cap\{s\eta\le v-u\}}&\left(A(x,Dv)-A(x,Du)\right)\cdot D\eta \ \dx\nonumber \\
&\le -\frac{1}{s}\int_{(U\setminus E)\cap\{ s\eta>v-u\}}\left(A(x,Dv)-A(x,Du)\right)\cdot (Dv-Du) \ \dx\le 0.
\end{flalign*}
The content of the previous display and the dominated convergence theorem allow concluding that 
\begin{flalign*}
\int_{U\setminus E}\left(A(x,Dv)-A(x,Du)\right)\cdot D\eta \ \dx\le 0
\end{flalign*}
and, since $u$ solves \eqref{A0}, we obtain
\begin{flalign*}
\int_{U\setminus E}A(x,Dv)\cdot D\eta \ \dx \le 0 
\end{flalign*}
for all non-negative $\eta\in W^{1,H(\cdot)}_{0}(U\setminus E)$, thus, recalling \eqref{meas}, $\mu(U\setminus E)=0$. Hence, $\mu(U)=0$, which means
\begin{flalign}\label{60}
\int_{U}A(x,Dv)\cdot Dw \ \dx =0 \ \ \mbox{for all} \ \ w\in W^{1,H(\cdot)}_{0}(U).
\end{flalign}
Now, set $\hat{A}(x,z):=-A(x,-z)$ and notice that, by very definition, the tensor $\hat{A}$ satisfies \eqref{A}. Consider $\hat{v}\in W^{1,H(\cdot)}(U)$ being a solution to the obstacle problem
\begin{flalign*}
\int_{\Omega}\hat{A}(x,D\hat{v})\cdot (Dw-D\hat{v}) \ \dx\ge 0 \quad\text{ for all }\ w\in \mathcal{K}_{-u}(U).
\end{flalign*}
 Proceeding as  before, we  conclude that $\hat{v}$ satisfies
\begin{flalign}\label{61}
\int_{U}-A(x,-D\hat{v})\cdot Dw \ \dx=\int_{U}\hat{A}(x,D\hat{v})\cdot Dw \ \dx=0,
\end{flalign}
for all $w\in W^{1,H(\cdot)}_{0}(U)$. Therefore, it only remains to show that $v=-\hat{v}=u$. If we define $\bar{w}:=v+\hat{v} \in W^{1,H(\cdot)}(U)$, therefore it is an admissible competitor in both \eqref{60} and \eqref{61}. Thus, testing them against $\bar{w}$, adding \eqref{60} and \eqref{61}  we obtain
\begin{flalign*}
0\le \nu\int_{U}\mathcal{V}(Dv,D\hat{v}) \ \dx \le\,{c} \int_{U}\left(A(x,Dv)-A(x,-D\hat{v})\right)\cdot(Dv+D\hat{v}) \ \dx =0,
\end{flalign*}
{cf.~\eqref{A-monotone}}. Therefore, $D(v+\hat{v})=0$ a.e. in $U$ and since $v+\hat{v}\in W^{1,H(\cdot)}_{0}(U)$, we have $v=-\hat{v}$ a.e. in $U$. Recalling that by {Theorem~\ref{T4}}, $-\hat{v}\le u\le v$ a.e. in $U$, from \eqref{60} and \eqref{61} we obtain that $u$ solves \eqref{A0} in $U$. Since $U$ is arbitrary, we can conclude that $u$ solves \eqref{A0} in $\Omega$, thus $E$ is removable.\qed

\end{document}